\documentclass[reqno]{amsart}

\usepackage{amsmath, amssymb, amsfonts, amsthm}
\usepackage{mathtools,bbm}
\usepackage{tikz}

\renewcommand{\theequation}{\arabic{section}.\arabic{equation}}

\theoremstyle{plain}
\newtheorem{theorem}{Theorem}[section] 

\newtheorem{corollary}[theorem]{Corollary}
\newtheorem{lemma}[theorem]{Lemma}
\newtheorem{proposition}[theorem]{Proposition}
\newtheorem{definition}[theorem]{Definition}

\newtheorem*{assumpintro}{Assumption}

\theoremstyle{remark}
\newtheorem{remark}{Remark}[section]

\usepackage{graphicx}
\usepackage{bm}
\usepackage{stmaryrd}
\usepackage{pdfpages}

\newcommand\C{\mathbb{C}}

\newcommand\R{\mathbb{R}}

\newcommand\cX{\mathcal{X}}
\newcommand\cR{\mathcal{R}}
\newcommand\cS{\mathcal{S}}
\newcommand\cT{\mathcal{T}}

\newcommand\cQ{\mathcal{Q}}

\newcommand\cD{\mathcal{D}}

\newcommand\diam{\operatorname{diam}}
\newcommand\dist{\operatorname{dist}}

\renewcommand\Re{\operatorname{Re}}
\renewcommand\Im{\operatorname{Im}}

\newcommand\osc{\operatorname{osc}}
\renewcommand\Pr[2]{\operatorname{Pr}[#1\,;#2]}

\newcommand\Unif{\textsc{Unif}}
\newcommand\UnifD{\textsc{Unif}($\delta$)}
\newcommand\LipK[2]{\textsc{Lip}($#1$,$#2$)}
\newcommand\ExpFat[2]{\textsc{ExpFat}($#1$,$#2$)}

\newcommand\xskip[1]{}

\begin{document}

\title[Strong RSW estimates for the FK-Ising model on s-embeddings]
{Strong RSW estimates for the FK-Ising model on~s-embeddings}

\author[Dmitry Chelkak]{Dmitry Chelkak}
\author[Yanqing Wei]{Yanqing Wei}

\address{University of Michigan, Department of Mathematics, Ann Arbor, MI 48109-1043, USA}

\email{\texttt{dchelkak@umich.edu}, \texttt{yqw@umich.edu}}



\begin{abstract} 
We adapt the `induction over scales' scheme from~\cite{duminil-manolescu-tassion-fractal} to the FK-Ising model on s-embeddings. Namely, we prove the existence of crossings of rough topological quadrilaterals with unfavorable boundary conditions, known as `strong' RSW, under the `uniformly bounded geometry' assumption (all edge lengths are comparable to~$1$ and all angles are bounded from below) on the tangential quads that form an s-embedding. This result has a number of standard consequences like quasi-multiplicativity of arm events and fractal properties of the clusters. It can be applied in a variety of setups including near-critical Ising models on regular grids below the correlation length scale by considering appropriate s-embedding thereof.
\end{abstract}

\maketitle

\section{Introduction}

\subsection{Context and related works} \label{subsec:context}
S-embeddings of weighted planar graphs is a framework suggested in~\cite{chelkak-icm18} and developed in~\cite{chelkak-semb} that allows one to study the (possibly, inhomogeneous) Ising model on a given graph by \emph{re-embedding} it into the complex plane in such a way that modern discrete complex analysis techniques can be applied to Kadanoff--Ceva fermionic observables. A model example of this procedure for the massive Z-invariant weights on isoradial graphs has been discussed in~\cite[Section~3.3]{chelkak-izyurov-mahfouf}. In this setup, the new embedding can be constructed explicitly using the rich algebraic structure of the model. A broader class of near-critical models on the square grid was later treated in~\cite{mahfouf-near-critical} using a dynamical approach: as one changes the interaction parameters starting from the critical homogeneous case, the evolution of {fermionic} correlators can be controlled via the resolvent identity, which allows one to construct a `nice' s-embedding of the near-critical Ising model under consideration.

RSW-type estimates, originally developed by Russo and Seymour--Welsh in the Bernoulli percolation context, are standard basic tools to control the geometry of clusters in a planar percolation model. In order to apply these techniques to the classical spin Ising model one passes to its random cluster or Fortuin--Kasteleyn (FK) representation and views it as a dependent percolation that encodes the spin correlations by connectivity properties of random clusters; e.g., see~\cite{grimmett-random-cluster-book}. For the critical FK-Ising model on~$\mathbb{Z}^2$, presumably the simplest setup comes from considering the probability of a horizontal crossing in a rectangle $[-m;m]\times[-n;n]$ with self-dual boundary conditions: wired on the vertical sides and free on the horizontal ones. One can use discrete complex analysis tools to prove (see~\cite[Section~6]{chelkak-smirnov-ising}) that this probability is uniformly bounded from below if so is the aspect ratio~$n/m$. Using monotonicity arguments and the FKG inequality it is not hard to deduce from this estimate a uniform lower bound for the probability that an annulus $[-3n;3n]^2\smallsetminus [-n;n]^2$ with \emph{free} boundary conditions contains a wired circuit; see~\cite[Proposition~2.10]{duminil-garban-pete} and, e.g.,~\cite[Section~5.6]{chelkak-semb}. By self-duality, this implies the polynomial decay of the magnetization at the center of a large box in the classical spin representation of the Ising model. In the s-embeddings context this strategy has recently been implemented in~\cite{mahfouf-RSW}; see Theorem~\ref{thm:RSW-Mahfouf} below. 

The RSW-type estimates described above are not sufficient for all the purposes: often, a difficulty comes from the need to know that a wired cluster \emph{touches} a nearby free boundary arc with probability that is uniformly bounded from below. The first result of this kind for the FK-Ising model on~$\mathbb{Z}^2$ was obtained in~\cite{duminil-hongler-nolin} for \emph{straight} boundaries: the probabilities of horizontal crossings in rectangles are uniformly bounded from below even if one imposes free boundary conditions on all four sides. However, even this is not enough if one wants, e.g., to prove that two distinct nearby clusters touch each other with good probability. (Such statements are important, e.g., for the proof of convergence of Ising and FK-Ising loop ensembles to CLEs; see~\cite{ising-5authors,benoist-hongler,kemppainen-smirnov-ii,garban-wu}.) This motivates replacing straight rectangles by arbitrary topological quads $(\cD;a,b,c,d)$, that is, simply connected discrete domains with four marked boundary points, for which one aims to prove the existence of a crossing between two opposite boundary arcs~$(ab)$ and~$(cd)$ under free boundary conditions. For the FK-Ising model on~$\mathbb{Z}^2$ (and, mutatis mutandis, on isoradial graphs) this has been done in~\cite{chelkak-duminil-hongler} based upon techniques developed in~\cite{chelkak-toolbox}. This proof relies upon very special properties of discrete holomorphic functions on rhombic lattices, which makes it impossible to adapt it to the s-embeddings context.

More recently, a much more powerful approach to such `strong' RSW estimates (that is, uniform lower bounds for the probability that a wired cluster touches a rough boundary arc with free boundary conditions) has been developed by Duminil-Copin, Manolescu, and Tassion in~\cite{duminil-manolescu-tassion-fractal}, where those were proved for all FK($q$) percolation models on~$\mathbb{Z}^2$ with $q\in [1;4)$, of which the FK-Ising model is a special case~$q=2$. In our paper we adapt this approach. 
It is worth saying that 
this adaptation is not straightforward for two reasons. First, \cite{duminil-manolescu-tassion-fractal} uses symmetries of~$\mathbb{Z}^2$ that we do not have in our setup. Second, \cite{duminil-manolescu-tassion-fractal} relies upon a crucial `positivity' property of (para)fermionic observables on straight cuts of~$\mathbb{Z}^2$ for which one needs to find a proper replacement. Such a replacement -- so-called \emph{`special cuts'} (or s-cuts) -- was proposed in~\cite[Section~5]{chelkak-semb}. Still, incorporating this notion into the framework of~\cite{duminil-manolescu-tassion-fractal} requires a non-trivial work; {see Section~\ref{subsec:organization} below for more comments.}

We do not detail standard applications of our main result and refer the reader to~\cite[Section~1.4]{duminil-manolescu-tassion-fractal} 
instead. Let us nevertheless mention a couple of them explicitly:
\begin{itemize}
\item Using appropriate re-embedding procedures (e.g., see~\cite{mahfouf-near-critical}), one can translate standard applications of the `strong' RSW to universal arm exponents \cite[Section~6.3]{duminil-manolescu-tassion-fractal} to various near-critical setups of interest on regular grids.
\item Following~\cite{ising-5authors,benoist-hongler,kemppainen-smirnov-ii,garban-wu}, one can upgrade very recent results of Park~\cite{park-maximal} on the convergence of Ising interfaces to SLE curves in a \emph{non-trivial complex structure} to the convergence to CLEs of full discrete loop ensembles.
\item Following~\cite{camia-feng-Ising}, one can use this convergence to CLE to deduce the convergence of multipoint spin correlations normalized by probabilities of one-arm events to CFT correlations in the same complex structure. {See also a very recent work~\cite{cahen-garban-sepuvelda} for an extended discussion of the links between convergence results for various objects arising in the critical planar Ising model.}
\end{itemize}

\begin{figure}
\begin{minipage}{0.47\textwidth}
\centering \includegraphics[clip, trim=9.8cm 16.1cm 6.36cm 7cm, width=\textwidth]{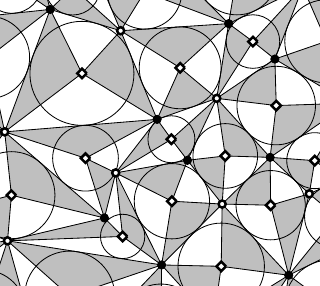}
\end{minipage}
\hskip 0.05\textwidth
\begin{minipage}{0.47\textwidth}
\includegraphics[clip, trim=4.5cm 16.7cm 9.8cm 4cm, width=\textwidth]{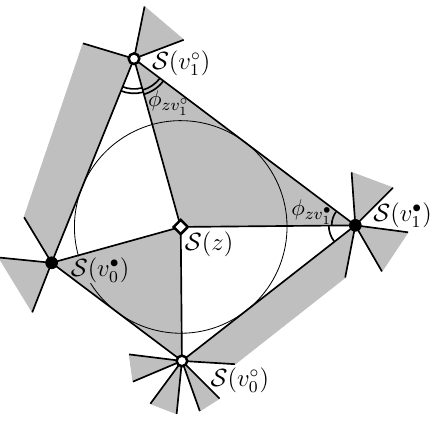}
\end{minipage}
\caption{\label{fig:Sdef} {\textsc{Left:} A piece of a tiling of the plane by tangential quads, called an s-embedding in~\cite{chelkak-semb}. The FK-Ising model is defined on the graph~$G^\bullet$ formed by black nodes and is dual to the similarly defined model on the dual graph~$G^\circ$ formed by white nodes. \textsc{Right:} Notation for half-angles~$\phi_{zv}\in (0;\pi)$ of a tangential quad.}}
\end{figure}
 
\subsection{Definition of the model and main results} An \emph{s-embedding} $\cS$ is a tiling of the plane by tangential quads. The vertices of such a tiling form two bipartite classes, $G^\bullet$ and~$G^\circ$, that can be viewed as planar graphs that are dual to each other; {see Fig.~\ref{fig:Sdef}.} Formally, $\cS$ is a mapping from the vertex sets of~$G^\bullet$ and~$G^\circ$ to the complex plane that assigns the coordinate~$\cS(v)\in\C$ to a vertex~$v\in G^\bullet\cup G^\circ$. For shortness, sometimes we will not distinguish~$v$ and~$\cS(v)$ in the notation if no confusion arises. One defines the FK-Ising model (Fortuin--Kasteleyn or random cluster representation of the Ising model) on~$G^\bullet$ as follows. First, given a tangential quad $z=(v_0^\bullet v_0^\circ v_1^\bullet v_1^\circ)$ of~$\cS$, let us denote (e.g., see~\cite[Eq.~(6.3)]{chelkak-icm18})
\begin{equation}
\label{eq:theta=}
\theta_z:=\arctan \left(\frac{\sin\phi_{zv^\bullet_0}\sin\phi_{zv^\bullet_1}}{\sin\phi_{zv^\circ_0}\sin\phi_{zv^\circ_1}}\right)^{\!\!1/2}\in\, (0;\tfrac\pi 2)\,,
\end{equation}
where $\phi_{zv}$ denotes the \emph{half-angle} of a tangential quad~$z$ at its vertex~$v$. Note that the set of tangential quads is in bijection with the edges of~$G^\bullet$. An FK-Ising model on (a finite part of) $G^\bullet$ is a dependent bond-percolation model in which the probability of a given configuration is proportional to its weight
\[
\prod_{z} \biggl(\frac{p_z}{1-p_z}\biggr)^{\mathbbm{1}[\mathrm{the~edge~of}~G^\bullet~\mathrm{in}~z~\mathrm{is~open}]}\times\, 2^{\#\mathrm{clusters}}\,,\qquad
p_z=\frac{2\tan\frac12\theta_z}{1+\tan\frac12\theta_z}\,.
\]

This definition can be viewed as a generalization of the critical Z-invariant Baxter weights on isoradial graphs for which the aforementioned tangential quads are rhombi with half-angles~$\theta_z$. The main advantage of this setup is that it removes all algebraic and combinatorial constraints on the class of weighted planar graphs carrying the Ising model that admit such a geometric representation. Note that interchanging the roles of~$G^\bullet$ and~$G^\circ$ results in replacing~$\theta_z$ by~$\theta_z^*=\frac\pi2-\theta_z$, which is equivalent to the self-duality identity $p_z p_z^*=2(1-p_z)(1-p_z^*)$. We refer the reader to~\cite{chelkak-semb} for a further discussion. 

In this paper we assume that the tiles of~$\cS$ have `uniformly bounded geometry':

\begin{assumpintro}[\Unif] There exists a constant~$c_0>0$ such that all edges of~$\cS$ are uniformly comparable to~$1$:
\[
c_0\,\le\,|\cS(v^\bullet)-\cS(v^\circ)|\,\le\,c_0^{-1}\ \ \text{if~$v^\bullet\in G^\bullet$ and~$v^\circ\in G^\circ$ are adjacent,}
\]
and all half-angles~$\phi_{zv}$ of quads in~$\cS$ are uniformly bounded from below by $c_0$. 
\end{assumpintro}
From here onwards we view the constant~$c_0$ as fixed once forever; all other constants that appear below may depend on~$c_0$ even if this is not mentioned explicitly. While the assumption \Unif\ may seem too restrictive, it is known to hold, e.g., for s-embeddings of various near-critical Ising models on regular grids below the correlation length scale; see~\cite[Section~3.3]{chelkak-izyurov-mahfouf} and~\cite{mahfouf-near-critical}. 

Denote by~$B_R=B_R(0)$ the Euclidean ball of radius~$R>0$ centered at the origin. Let~$\mathbb{P}^0_U$ be the probability measure on FK-Ising configurations in a domain~$U$ with free boundary conditions and~$\mathrm{Circ}(A)$ be the event that such a configuration contains a cluster that encircles the set~$A$. The following theorem is a particular case of a more general result (see Remark~\ref{rem:RSW-Mahfouf} below) obtained by Mahfouf~\cite{mahfouf-RSW}.

\begin{theorem}[RSW,~\cite{mahfouf-RSW}]\label{thm:RSW-Mahfouf} There is a constant $p_\mathrm{circ}=p_\mathrm{circ}(c_0)>0$ such that
\begin{equation}
\label{eq:intro-circ}
\mathbb{P}^0_{B_R\smallsetminus B_{\frac12R}}\big[\mathrm{Circ}(B_{\frac12R})\big]\,\ge\,p_\mathrm{circ}\,.
\end{equation}
\end{theorem} 
\begin{remark}\label{rem:RSW-Mahfouf}
In fact, Theorem~\ref{thm:RSW-Mahfouf} holds in a much more general setup. Namely, it is proved in~\cite{mahfouf-RSW} on all scales~$R\ge C\delta'$ under  so-called assumptions \LipK{\kappa}{\delta}\ and \ExpFat{\delta}{\delta'}; see~\cite{chelkak-semb} for definitions. This is the maximal setup in which discrete complex analysis techniques developed in~\cite{chelkak-semb,CLR1} can be directly applied. Assumption \textsc{Lip}\ is conceptual and cannot be dropped (e.g., see a discussion after~\cite[Eq.~(1.12)]{chelkak-semb}). The status of the `technical' regularity assumption~\textsc{ExpFat} on~$\cS$ is less clear: to the best of our knowledge, it is not known whether it is indeed necessary for the existence of circuits~\eqref{eq:intro-circ} on scales~$R\ge C\delta$ or not. 
\end{remark}
\begin{remark}\label{rem:RSW-Mahfouf-alt} In the setup of this paper -- i.e., under assumption~\Unif\ -- Theorem~\ref{thm:RSW-Mahfouf} can also be proved similarly to~\cite[Theorem~1.3 and Corollary~1.4]{chelkak-semb}. Namely, it is well known 
that the estimate~\eqref{eq:intro-circ} follows from the existence of crossings in sufficiently long quadrilaterals with \emph{self-dual} (i.e., wired/free/wired/free) boundary conditions and nice boundaries. If the two wired parts of such a quadrilateral are `special' cuts (see Section~\ref{subsec:cuts} below for a definition), then a near-to-boundary analysis of fermionic observables similar to the one presented in Section~\ref{subsec:step(vii)} of this paper allows one to adapt the proof of~\cite[Theorem~1.3]{chelkak-semb} to our, more general, setup with minimal changes only. See also~\cite[Section~5.5.1]{chelkak-semb}.
\end{remark}

The main result of our paper is the following theorem:

\begin{theorem}\label{thm:main-theorem}
Consider the Ising model on an s-embedding~$\cS$ satisfying the assumption~\Unif. There exists a constant~$p_*=p_*(c_0)>0$ such that for each bounded simply connected domain~$\cD$ on~$\cS$ that contains the origin one has
\[
\mathbb{P}^0_\cD\big[B_{\frac12R}\xleftrightarrow{B_{5R}\,}\partial\cD\big]\,\ge\, p_*\,,
\]
where $R=\dist(0,\partial\cD)$ and $A\xleftrightarrow{U}B$ denotes the event that there exists an open path connecting two sets~$A$ and~$B$ inside~$U\cap\cD$. 
\end{theorem}

Due to the FKG inequality and Theorem~\ref{thm:RSW-Mahfouf} our result implies that macroscopic FK-Ising clusters touch an arbitrarily rough free boundary with probability that is uniformly bounded from below. It is also known that it has the following corollary; see Section~\ref{subsec:context} above for a discussion of the importance of such a uniform estimate.
\begin{corollary}[`strong' RSW]
For each $M>0$, there is a constant \mbox{$p_0(M)\in (0;1)$} that also depends on~$c_0$ such that for any discrete quadrilateral $\left(\cD;a,b,c,d\right)$ and any boundary conditions $\xi$ on~$\partial\cD$,
\begin{itemize}
\item if 
$\mathrm{EL}_{\cD}\left[\left(ab\right),\left(cd\right)\right] \leq M$\,, then $\mathbb{P}_{\cD}^{\xi}[\left(ab\right)\xleftrightarrow{\cD}\left(cd\right)] \geq p_0(M)$;
\item if $\mathrm{EL}_{\cD}\left[\left(ab\right),\left(cd\right)\right] \geq M^{-1}$, then $\mathbb{P}_{\cD}^{\xi}[\left(ab\right)\xleftrightarrow{\cD}\left(cd\right)] \leq 1-p_0(M)$,
\end{itemize}
where~$(ab)$, $(bc)$, $(cd)$, and~$(da)$ stand for four boundary arcs of~$\cD$ and we denote by $\mathrm{EL}_\cD[(ab),(cd)]$ the extremal length of the family of curves connecting the opposite boundary arcs~$(ab)$ and~$(cd)$ inside~$\cD$.
\end{corollary}
\begin{proof}
See~\cite[Section~4]{duminil-manolescu-tassion-fractal}, in which a similar result~\cite[Theorem~1.2]{duminil-manolescu-tassion-fractal} for the FK($q$), $q\in [1;4)$, percolation model on the square grid is shown to be a corollary of~\cite[Proposition~1.3]{duminil-manolescu-tassion-fractal}, which is similar to our Theorem~\ref{thm:main-theorem}.
\end{proof}

{

\subsection{Organization of the paper and key ingredients of the proof} \label{subsec:organization} We recall the key geometric notion of a space-like \emph{s-surface} $\Theta\subset\R^{2,1}$ associated with an s-embedding and that of `special cuts' in Sections~\ref{subsec:s-surfaces} and~\ref{subsec:cuts}, respectively. The technical cornerstone of our proof is matching upper and lower bounds for the expected number of one-arm events on the aforementioned `special cuts' that are formulated in Section~\ref{subsec:one-arm-bounds}. The proofs of these estimates rely upon discrete complex analysis tools that are very specific to the planar Ising model and are postponed until Section~\ref{sec:fermionic-proofs}. The upper bound follows from the known a priori regularity theory of s-holomorphic observables in a relatively straightforward way; see Section~\ref{subsec:regularity-of-F}, The proof of the lower one is much harder. For this, we significantly generalize ideas from~\cite[Section~5]{chelkak-semb} to the case of \emph{non-flat} s-surfaces in Sections~\ref{subsec:sketch-lower-bound}--\ref{subsec:step(vii)}. 

If we worked on~$\mathbb{Z}^2$, these matching upper and lower bounds would be equivalent to the fact that the boundary one-arm exponent equals~$\frac12$, from which one can run the renormalization argument from~\cite{duminil-manolescu-tassion-fractal} verbatim. However, as we already mentioned above, the lack of translation invariance creates another problem related to a technical trick used in~\cite[Definition~5.4]{duminil-manolescu-tassion-fractal}. We replace this trick by finding sufficiently \emph{many} special Dobrushin subdomains of a given discrete domain -- that we introduce in Section~\ref{subsec:bridge-shore} and call `bridge-shore' subdomains -- instead of a \emph{single} Dobrushin subdomain used in~\cite[Section~5]{duminil-manolescu-tassion-fractal}. This makes the whole argument more robust; see the proof of Proposition~\ref{prop:lower-bound-w/o-far} in Section~\ref{subsec:first-moment} for details. 

Since our `bridge-shore' subdomains are cut by `special cuts' rather than by straight lines, the proof of their existence also requires some effort. This is done in Section~\ref{subsec:fjords} using some facts about the mutual geometry of such `special cuts' that are collected in Appendix. The proof of our main result, Theorem~\ref{thm:main-theorem}, is assembled from all these ingredients in Section~\ref{subsec:proof-main-thm} following~\cite{duminil-manolescu-tassion-fractal}.

\subsection*{Acknowledgements} The authors are grateful to Yu Feng and S.C.Park for helpful discussions and remarks.

\subsection*{AI disclosure} We did not use generative AI when working on this paper.}

\section{Special cuts and bridge-shore subdomains}
\label{sec:cuts-and-fjords}
\setcounter{equation}{0}

\subsection{S-surfaces and S-graphs associated with s-embeddings}\label{subsec:s-surfaces} 
As we already mentioned above, the main `geometric' concept that we need in order to adapt the renormalization scheme from~\cite{duminil-manolescu-tassion-fractal} to our setup is the notion of~\emph{special cuts} introduced in~\cite[Section~5]{chelkak-semb}. In turn, its definition relies upon a more fundamental construction: that of \emph{s-surfaces} associated with s-embeddings. For convenience of the reader, we recall the definition of s-surfaces and comment on it in this section, and postpone the definition of special cuts to Section~\ref{subsec:cuts}. In order to clarify these definitions, in these two sections we use the notation~$\cD_\C\subset\C$ for the planar domain covered by~$\cD$. However, we often use the same notation for vertices and their positions in the complex plane if no confusion arises.

In what follows, we use the notation~$G^\bullet$ for the set of primary vertices, $G^\circ$ for the set of dual vertices and~$\diamondsuit(G)$ for the set of quads of an s-embedding~$\cS$. One of the key ideas of~\cite{chelkak-semb} is that the `right' geometry of the Ising model with interactions~\eqref{eq:theta=} is provided \emph{not} by~$\cS$ itself but rather by its lift to a certain piecewise linear surface~$\Theta$ in the Minkowski space~$\R^{2,1}$. In particular, one can check that the quantity~\eqref{eq:theta=} is nothing but the ratio of the~$\|\cdot\|_{2,1}$-distances between the positions of the dual vertices~$v_0^\circ$ and~$v_1^\circ$, and of the primary vertices~$v_0^\bullet$ and~$v_1^\bullet$ on~$\Theta$, which makes this representation invariant under isometries of~$\R^{2,1}$; see~\cite[Remark~1.2]{chelkak-semb}.

Let us now recall the construction of the surface~$\Theta$. We first define a function $\cQ:G^\bullet\cup G^\circ\to\R$ by requiring that
\[
\cQ(v^\bullet)-\cQ(v^\circ)=|\cS(v^\bullet)-\cS(v^\circ)|\ \ \text{for adjacent~$v^\bullet\in G^\bullet$ and~$v^\circ\in G^\circ$.}
\]
(This definition is consistent around each tangential quad.) Moreover, we can consistently extend the definition of~$\cQ$ to centers of quads by requiring that
\[
\cQ(v)-\cQ(z)=\pm r_z\cot \phi_{zv} \ \  \text{for adjacent~$v\in  G^\bullet\cup G^\circ$ and~$z\in\diamondsuit(G)$,}
\]
where~$\phi_{zv}$ is the half-angle of the tangential quad~$z$ at one of its vertices~$v$, and one chooses the `$+$' sign if~$v\in G^\bullet$ and the `$-$' sign if~$v\in G^\circ$. Then, we extend~$\cQ^\bullet$ to a function~$\cQ:\cD_\C\to\R$ by requiring that~$\cQ$ is linear on each of the four triangles that constitute a tangential quad. Finally, we set
\begin{equation}
\label{eq:Theta-def}
\Theta:=\{(z;\cQ(z)),\ z\in\cD_\C\}\subset \C\times\R\cong \R^{2,1}
\end{equation}
and call $\Theta$ the \emph{real s-surface} associated with the s-emebdding~$\cS$. Recall that one can view s-embeddings as a particular case of more general t-embeddings arising in the bipartite dimer model context; see~\cite[Section~8.2]{CLR1} for details. From this perspective, real s-surfaces~$\Theta$ are obtained from t-surfaces of the associated bipartite dimer model by forgetting their microscopic fourth coordinate of~$\R^{2,2}$.

To define the main combinatorial tool of our paper -- special cuts of discrete domains -- it is also convenient to slightly modify the definition of~$\Theta$. It is easily follows from the property~\Unif\ of~$\cS$ that there are constants~$\kappa=\kappa(c_0)<1$ and~$C=C(c_0)>0$ such that~$\Theta$ is a $\kappa$-Lipschitz surface on scales above~$C$, that is, one has~$|\cQ(z')-\cQ(z)|\le\kappa|z-z'|$ for all~$z,z'\in\cD_\C$ such that~$|z'-z|\ge C$. However, this property clearly does not hold on small scales, which creates a technical inconvenience. (Recall that, by definition, the edges of~$\Theta$ connecting a primary and an adjacent dual vertices are light-like.) To avoid this, we introduce a slightly modified version~$\cQ^\bullet:\cD_\C\to\R$ of the function~$\cQ$ as follows:
\begin{itemize}
\item set~$\cQ^\bullet=\cQ$ at all primary vertices, all centers of quads of~$\cS$, and on all segments connecting these centers with primary vertices;
\item note that the property~\Unif\ implies that $\cQ$ is a $\kappa$-Lipschitz function on each polygonal contour~$v_1^\circ\sim z_1\sim  v_2^\circ\sim  ... \sim z_n \sim v_{n+1}^\circ=v_1^\circ$ surrounding a primary vertex of~$\cS$, where~$\kappa<1$ is a constant that depends on~$c_0$ only;
\item let~$\cQ^\bullet\ge\cQ$ be an arbitrary~$\kappa$-Lipschitz extension of~$\cQ$ inside such polygons. 
\end{itemize}
Similarly to~\eqref{eq:Theta-def} we set~$\Theta^\bullet=\{(z,\cQ^\bullet(z)),\,z\in\cD_\C\}$. By construction,~$\Theta^\bullet$ lies in an~$O(1)$ vicinity above~$\Theta$. One can also apply a similar procedure to dual vertices and obtain a modified function~$\cQ^\circ\le \cQ$ and the corresponding surface $\Theta^\circ=\{(z,\cQ^\circ(z)),\,z\in\cD\}$ that lies in a~$O(1)$ vicinity below~$\Theta$.

Below we use the notation~$\dist(z',z)=|z'-z|$ for the usual \emph{Euclidean} distance between two points~$z,z'\in\cD$. In several situations it is also useful to work with a pseudo-distance
\[
\|z'-z\|_{2,1}\,\:=\,\big(|z'-z|^2-|\cQ^\bullet(z')-\cQ^\bullet(z)|^2\big)^{\frac12},\ \ z,z'\in\cD_\C,
\]
which is nothing but the Minkowski~$\|\cdot\|_{2,1}$-distance between the lifts of the points~$z,z'$ onto the modified s-surface~$\Theta^\bullet$. It is easy to see that
\begin{equation}
\label{eq:dist-2,1-comparable}
\|z'-z\|_{2,1}\,\le\,\dist(z,z')\,\le\,(1-\kappa^2)^{-1/2}\|z'-z\|_{2,1}
\end{equation}
for all~$z,z'\in\cD_\C$ as the modified s-surface~$\Theta^\bullet$ is $\kappa$-Lipschitz on \emph{all} scales.

It is also useful to recall the notion of \emph{S-graphs} associated with an s-embedding~$\cS$ that appeared in~\cite{chelkak-semb}. Given~$\alpha\in\C$, $|\alpha|=1$, the position of a vertex~$v\in G^\bullet\cup G^\circ$ in the corresponding S-graph is given by~$\cS(v)-i\alpha\cQ(v)$. (Note that we use a nonstandard convention for the value of~$\alpha$ as compared to~\cite{chelkak-semb,CLR1}.) One can view this as the projection of the s-surface~$\Theta$ onto the horizontal plane along the light-like direction~$(i\alpha;1)\in\C\times\R\cong\R^{2,1}$. The images of tangential quads~$\cS(v_0^\bullet)\cS(v_0^\circ)\cS(v_1^\bullet)\cS(v_1^\circ)$ in an S-graph are non-convex quadrilaterals. E.g., if~$\alpha=1$, then it is easy to see that both dual vertices lie above both primary ones since
\[
\Im(\cS(v^\circ))-\cQ(v^\circ)\,\ge\,\Im(\cS(v^\bullet))-\cQ(v^\bullet)\ \ \text{for}\ \  v^\circ\sim v^\bullet. 
\]
If $\cS(v^\bullet)-\cS(v^\circ)\in i\alpha\R_+$, then the edge~$(v^\circ v^\bullet)$ of~$\diamondsuit(G)$ degenerates in the corresponding S-graph. For simplicity, we will always assume that such degeneracies do not happen, which can be achieved by a small perturbation of the direction~$\alpha$.

\begin{remark}
\label{rem:duality} The Ising model on s-embeddings is self-dual. However, the definition of s-surfaces and S-graphs is not fully symmetric with respect to swapping the roles of primary and dual vertices. Namely, this results in changing the sign of the function~$\cQ$ (that is, to reflecting~$\Theta$ with respect to the horizontal plane) and hence to replacing $\alpha$ with~$-\alpha$ in the construction of S-graphs~$\cS-i\alpha\cQ$.
\end{remark}

\begin{figure}
\begin{minipage}{0.5\textwidth}
\includegraphics[clip, trim=4.5cm 16.5cm 7.7cm 4.5cm, width=\textwidth]{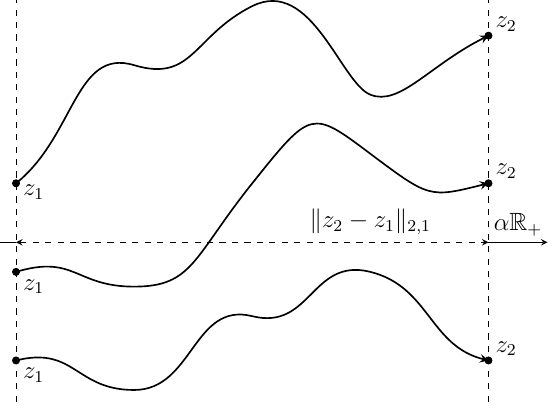}
\end{minipage}
\hskip 0.05\textwidth
\begin{minipage}{0.44\textwidth}
\includegraphics[clip, trim=4.6cm 9.2cm 9.2cm 12.65cm, width=\textwidth]{strongRSW_fig_cuts_geom.pdf}
\end{minipage}
\caption{{An illustration of the mutual geometry of `special cuts'. We need to use such cuts instead of straight lines due to the `positivity'
of fermionic observables on them (see Proposition~\ref{prop:observables-boundary}(iii)). 
\textsc{Left:} A possible configuration of three `special cuts' going in the \emph{same} direction $\alpha\R_+$ at different heights. The~$\|\cdot\|_{2,1}$ distances between the lifts of points lying on the same cut onto~$\Theta^\bullet\subset\R^{2,1}$ are equal to the distance between their projections onto the line~$\alpha\R$. 
\textsc{Right:} `Special cuts' emanating from the same point do not intersect each other and form a `wiggled' version of the polar coordinate system; see Lemma~\ref{lem:polar-coordinates} for a quantitative estimate.}
\label{fig:cuts-geom}}
\end{figure}

\subsection{Special cuts} \label{subsec:cuts}

The following definition was introduced in~\cite[Section~5.2]{chelkak-semb} in a slightly different and technically more complicated form.

\begin{definition}\label{def:s-cut}
Let~$\alpha\in\C$, $|\alpha|=1$, be such that the S-graph~$\cS-i\alpha\cQ$ is non-degenerate and~$\mu\in\R$. Denote
\[
\cD_\C^{[\alpha,\mu+]}:=\{z\in\cD_\C:\Im[\overline\alpha z]>\cQ^\bullet(z)+\mu\}\,.
\]
We call the curve~$\partial\cD_\C^{[\alpha,\mu+]}$ the `special cut' of~$\cD_\C$ going in the direction~$\alpha$ at height~$\mu$. Further, the corresponding special cut of the discrete domain~$\cD$ is a simple path on the graph~$G^\circ\cup\diamondsuit(G)$ that is the boundary of the main connected component of the subset of primary vertices of~$\cD$ that lie in $\cD_\C^{[\alpha,\mu+]}$. We use the same name `special cut' for subsegments of these curves and discrete paths. 
\end{definition}
\begin{remark} 
A useful way of thinking about special cuts is to view them as appropriate discretizations of \emph{straight lines} $\alpha (t+i\mu)$, $t\in\R$, \emph{on the S-graph}~$\cS-i\alpha \cQ$. Note that one should consider different charts~$z-i\alpha Q^\bullet(z)$, $z\in\cD_\C$, for different~$\alpha$. {The mutual geometry of these cuts in the original complex plane is less straightforward to control but they still have some nice properties that allow us to work with them as if they were straight lines; see Lemma~\ref{lem:polar-coordinates} below and Fig.~\ref{fig:cuts-geom} for an illustration.}
\end{remark}

It is easy to see that
\[
\|z'-z\|_{2,1}\,=\,\big|\Re [\overline\alpha z']-\Re [\overline\alpha z]\big|\ \ \text{for all $z,z'\in \gamma:=\partial\cD_\C^{[\alpha,\mu+]}$.}
\]
Moreover, since~$\cQ^\bullet:\cD_\C\to\R$ is a $\kappa$-Lipschitz function with~$\kappa<1$, the curve $\gamma$ is the graph of a~$\kappa(1-\kappa^2)^{-1/2}$-Lipschitz function $y=y(x)$, where $y=\Im [\overline\alpha z]$ and $x=\Re [\overline \alpha z]$. If~$(ab)\subset\gamma$ is a special cut, then we typically use the~$\|\cdot\|_{2,1}$-distance to parameterize~$\gamma$ and call the quantity~$\|a-b\|_{2,1}$ the \emph{$\|\cdot\|_{2,1}$-length} of~$\gamma$.

We also need a notion of a special cut going from a point to another point rather than in a given direction. This notion is justified by the following lemma.
\begin{lemma}
\label{lem:s-cut-from-a-to-b} For each distinct points~$z_0,z\in\cD_\C$ there is a unique special cut~$\gamma$ of~$\cD_\C$ such that both~$z_0,z\in\gamma$ and~$\Re[\overline\alpha z_0]<\Re[\overline\alpha z]$, where~$\alpha$ is the direction of~$\gamma$.
\end{lemma}
\begin{proof}
Recall that~$\cQ^\bullet(z)$ is a $\kappa$-Lipschitz, $\kappa<1$, function of~$z\in\cD$. The condition that there is a special cut going in the direction~$\alpha$ and passing through both~$z_0$ and~$z$ can be written as
\[
\Im[\overline\alpha (z-z_0)]=\cQ^\bullet(z)-\cQ^\bullet(z_0).
\]
This equation has two solutions~$\alpha\in\C$, $|\alpha|=1$, and exactly one of them satisfies the condition~$\Re[\overline\alpha (z-z_0)]>0$. (The other solution is the direction of the special cut going from~$z$ to~$z_0$.) 
\end{proof}
\begin{remark}
\label{rem:s-cut-from-a-to-b} Let us emphasize that, in general,~$\alpha\ne \exp[i\arg(z-z_0)]$ and that the special cuts from~$z_0$ to~$z$ and from~$z$ to~$z_0$ do \emph{not} coincide. Moreover, these two special cuts can deviate by the distance of order~$\|z-z_0\|_{2,1}$ from each other.
\end{remark}
It follows from Lemma~\ref{lem:s-cut-from-a-to-b} that special cuts of~$\cD_\C$ emanating from the same point~$z_0$ in directions~$e^{i\phi}$, $\phi\in\R/2\pi\mathbb{Z}$, are disjoint. We provide a quantitative version of this statement in Lemma~\ref{lem:polar-coordinates} and Corollary~\ref{cor:cuts-do-not-screen-each-other}. This allows us to view such a family of special cuts as a `wiggled version' of the usual straight rays emanating from~$z_0$.

In our paper we also need the notion of \emph{dual special cuts,} which are nothing but special cuts for the dual Ising model on~$G^\circ$. Their definition is fully similar to Definition~\ref{def:s-cut}: we denote 
\[
\cD_\C^{[\alpha,\mu-]}:=\{z\in\cD_\C:\Im[\overline\alpha z]<\cQ^\circ(z)+\mu\}\,.
\]
and call~$\partial\cD_\C^{[\alpha,\mu-]}$ the dual special cut of~$\cD_\C$ going in the direction~$-\alpha$ at height~$-\mu$. The corresponding special cut of~$\cD$ is a simple path on the graph~$G^\bullet\cup\diamondsuit(G)$ that bounds the main connected component of the subset~$G^\circ\cap\cD_\C^{[\alpha,\mu-]}$ of dual vertices. We use the notation~$(ba)^*$ for a dual special cut going from~$b$ to~$a$.

Recall that~$\cQ^\circ\le\cQ\le\cQ^\bullet$. Therefore, the special cut~$\partial\cD_\C^{[\alpha,\mu+]}$ and the dual special cut~$\partial\cD_\C^{[\alpha,\mu-]}$ do not intersect and stay~$O(1)$ close to each other. Now assume, without loss of generality, that a discrete special cut~$(ab)$ goes in the positive direction. It follows from the construction that each primary vertices of~$\cD$ lies either above~$(ab)$ or below or at the corresponding dual special cut~$(ba)^*$, except maybe few exceptions in a~$O(1)$ vicinity of these cuts. These possible exceptions are \emph{not} relevant for our arguments. Because of this, in what follows we slightly oversimplify the picture and think about dual special cuts as primary special cuts that are tracked in the opposite direction.

Also, in what follows we do not pay attention to a subtle difference between continuous curves~$\partial\cD_\C^{[\alpha,\mu]}$ -- that we called special cuts of~$\cD_\C$ above -- and the corresponding special cuts of a discrete domain~$\cD$. The latter discrete paths stay at a~$O(1)$ distance from the former continuous curves and we will only be using geometric properties of these curves on larger scales, which are then inherited by the corresponding discrete paths.

\subsection{Expected number of boundary one-arm events on special cuts} \label{subsec:one-arm-bounds}
 Recall that a \emph{Dobrushin domain}~$\Omega$ is a discrete simply connected domain with two marked boundary points $a,b$ such that the Ising model has free boundary conditions on the boundary arc~$(ab)_{\Omega}$ and wired boundary conditions on the complementary arc~$(ba)_{\Omega}$. Combinatorially, the arc~$(ab)_{\Omega}$ is composed of half-quads of the form $(v_-^\circ(z) v^\bullet(z) v_+^\circ(z))$ instead of full quads, the arc~$(ba)_{\Omega}$ consists of half-quads $(v^\bullet_-(z)v^\circ(z)v^\bullet_+(z))$, and~$a,b$ can be viewed as edges of the graph~$\diamondsuit(G)$. 

The importance of special cuts for the analysis of the Ising model on s-embeddings is justified by an observation made in~\cite[Section~5.3]{chelkak-semb} that boundary values of the so-called fermionic observables in Dobrushin domains satisfy a certain positivity property if~$(ab)$ is a special cut. This observation allows one to use the discrete complex analysis machinery developed in~\cite{chelkak-semb} in order to estimate the expected number of boundary one-arm events on such {boundary arcs~$(ab)_\Omega=(ab)$.} It is worth mentioning that we are not aware of similar techniques that could be applied to discretizations of \emph{straight} cuts in our setup.

We now formulate two results that can be obtained by these discrete complex analysis techniques and postpone their (rather technical, especially that of Proposition~\ref{prop:lower-bound}) proofs until Section~\ref{sec:fermionic-proofs}. Let us nevertheless mention here that our arguments contain several simplifications as compared to~\cite{chelkak-semb}, partly inspired by~\cite{mahfouf-phd}. In particular, this allows us to go beyond the setup of `almost horizontal' s-surfaces considered in~\cite{chelkak-semb}.

We denote by~$\mathbb{P}^{0/1}_\Omega$ the probability measure on configurations in a Dobrushin domain~$\Omega$, where the superscript~$0/1$ emphasizes boundary conditions: free on the boundary arc~$(ab)_{\Omega}$ and wired on~$(ba)_{\Omega}$.

\begin{proposition}\label{prop:upper-bound} Let~$\Omega$ be a Dobrushin domain such that its free boundary arc $(ab)_\Omega=(ab)$ is a special cut of { $\|\cdot\|_{2,1}$-}length~$L$. {If a constant $\eta>0$} is such that 
\[
\dist(z,(ba)_{\Omega})\ge {\eta}\cdot \min\{\|z-a\|_{2,1},\|z-b\|_{2,1}\}\ \ \text{for all~$z\in(ab)$,}
\]
then
\[
\sum_{z\in (ab)}\mathbb{P}^{0/1}_\Omega(v^\bullet(z)\leftrightarrow (ba)_{\Omega}) = O({\eta^{-\frac12}}L^{\frac12}),
\]
where the implicit constant in this upper bound {depends on~$c_0$ only.}
\end{proposition}

\begin{proof} See Section~\ref{subsec:regularity-of-F}. 
\end{proof}

\begin{remark}
We will use this proposition in two different setups. In the first, the boundary of the Dobrushin domain will consist of the special cut~$\gamma=(ab)$, two segments perpendicular to the \emph{direction} of~$\gamma$ having Euclidean length~$\|b-a\|_{2,1}$, and a translation of~$\gamma$. Since~$\gamma$ is a graph of a $\kappa(1-\kappa^2)^{-1/2}$-Lipschitz function, for such a domain the assumption of Proposition~\ref{prop:upper-bound} holds with~$\eta>0$ that depends on~$c_0$ only. The second setup is described in the forthcoming Section~\ref{subsec:bridge-shore}. Due to Lemma~\ref{lem:ramps-not-screened} given below, in this case the constant~$\eta>0$ also depends on~$c_0$ only. Thus, in principle one can formulate Proposition~\ref{prop:upper-bound} without mentioning~$\eta$ at all. We decided to include this parameter in the statement so as to make the dependence of many implicit constants from each other more transparent.
\end{remark}

\begin{proposition}\label{prop:lower-bound} Let~$\Omega$ be a Dobrushin domain such that its free boundary arc $(ab)_\Omega=(ab)$ is a special cut of { $\|\cdot\|_{2,1}$-}length~$L$. {Assume also that~$\Omega$ does not contain exceptional vertices of~$G^\bullet$ that lie below the special cut~$(ab)$ but above the dual special cut~$(ba)^*$ (see a discussion in Section~\ref{subsec:cuts} above).} Then,
\begin{equation}
\label{eq:prop-lower-bound}
\sum_{z\in (ab)}\mathbb{P}^{0/1}_\Omega(v^\bullet(z)\leftrightarrow (ba)_{\Omega}) \ge \mathrm{cst}\cdot L^{\frac12},
\end{equation}
where~$\mathrm{cst}>0$ is a constant that depends on~$c_0$ only.
\end{proposition}

\begin{proof} See Sections~\ref{subsec:sketch-lower-bound}--\ref{subsec:step(vii)}. 
\end{proof}

\begin{remark}
\label{rem:lower-bound} In fact, we will use Proposition~\ref{prop:lower-bound} in the \emph{dual} setup, that is, for Dobrushin domains with wired boundary arcs~$(ba)_\Omega=(ba)^*$ being \emph{dual} special cuts. (See Section~\ref{subsec:lower-bound-w/o-legs} where this dual version of Proposition~\ref{prop:lower-bound} is used to prove Proposition~\ref{prop:lower-bound-w/o-legs} given below.) We chose to formulate it in terms of the primary model so as to put it in the context similar to Proposition~\ref{prop:upper-bound} and~\cite[Section~5]{chelkak-semb}.
\end{remark}

\subsection{`Bridge-shore' subdomains}\label{subsec:bridge-shore} 
We are now ready to introduce the main geometric definition of our paper: special subdomains of a simply connected domain~$\cD$ that we call \emph{`bridge-shore' subdomains}, which play the same role as that of `corner' domains in~\cite[Section~5.3]{duminil-manolescu-tassion-fractal}. A `bridge-shore' subdomain~$\Omega$ is defined by a pair of points~$a,b\in\cD$ satisfying several conditions listed in Definition~\ref{def:bridge} below. In what follows we will always view~$\Omega$ as a Dobrushin domain, with \emph{wired} boundary conditions on the `bridge' part of its boundary and free boundary conditions on the complementary arc, which consists of two `ramps' $(aa_0)$ and~$(b_0b)$ and a `shore', by which we mean the relevant part of~$\partial\cD$. See {Fig.~\ref{fig:bridge-shore}} for an illustration.

\begin{figure}
\centering 
\includegraphics[clip,trim=4.7cm 17cm 8.8cm 4.8cm,width=0.56\textwidth]{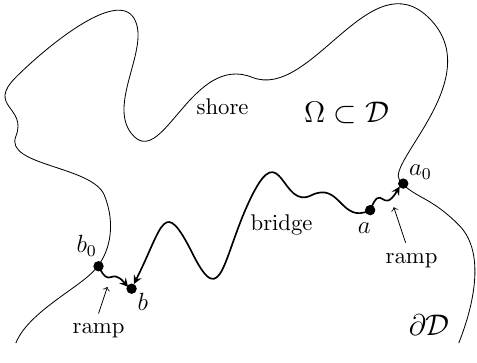}
\caption{{A schematic drawing of a `bridge-shore' subdomain. The `bridge' and two `ramps' are special cuts oriented from~$a$ or towards~$b$; the `shore' part of~$\partial\Omega$ is the boundary arc~$(a_0b_0)_\cD$.}
\label{fig:bridge-shore}}
\end{figure}

\begin{definition}
\label{def:bridge} Let~$a,b\in\cD$ and~$\rho>0$ be a {sufficiently} small parameter; {see Remark~\ref{rem:ramps-are-far} below.} We say that the special cut~$(ab)$ is a $\rho$-`bridge' in~$\cD$ if the following conditions hold:
\begin{itemize}
\item $(ab)$ is contained in~$\cD$ and, moreover, lies in the main (i.e., containing the origin) connected component of the Euclidean ${ \frac12\rho}\|b-a\|_{2,1}$-interior of~$\cD$;
\item one has~$\dist(a,\partial\cD)\le \rho\|b-a\|_{2,1}$ and the same for~$\dist(b,\partial\cD)$.
\end{itemize}
Further, let~$a_0,b_0\in\partial\cD$ be the $\|\cdot\|_{2,1}$-closest to~$a$ and $b$, respectively, boundary points of~$\cD$. We call two special cuts~$(aa_0)$ and~$(b_0b)$ the `ramps' of the `bridge'~$(ab)$. Finally, we say that~$a$ and~$b$ define a `bridge-shore' subdomain of~$\cD$ if 
\begin{itemize}
\item the part of~$\cD$ that lies to the {left} of {$(b_0b)\cup (ba)^*\cup(aa_0)$, which we call~$\Omega$, does not contain the origin. (Recall that $(ba)^*$ is the dual special cut going from~$b$ to~$a$, which always stays $O(1)$--close to the primary special cut~$(ab)$.)}
\end{itemize}
\end{definition}

\begin{remark}
\label{rem:ramps-are-far} 
{It follows from Lemma~\ref{lem:polar-coordinates} that} two special cuts emanating from the same point cannot intersect, which implies that the `ramps' $(aa_0)$ and $(b_0b)$ never intersect the `bridge' $(ab)$; see Lemma~\ref{lem:ramps-not-screened} below for a stronger version of this statement. It also easily follows from~\eqref{eq:dist-2,1-comparable} that~$\diam((aa_0))\le (1-\kappa^2)^{-1/2}\rho \dist(a,b)$ and similarly for~$(b_0b)$. Therefore, if~$\rho\le \frac13(1-\kappa^2)^{1/2}$, then
\begin{equation}
\label{eq:ramps-are-far}
\dist((aa_0),(b_0b))\,\ge\,\tfrac13\dist(a,b)\,\ge\,\tfrac13\|b-a\|_{2,1}\,.
\end{equation}
In particular, this means that the two `ramps'~$(aa_0)$ and~$(b_0b)$ do not intersect each other and hence~$(b_0b)\cup(ba)^*\cup(aa_0)$ is a simple curve that cuts~$\cD$ into two parts.
\end{remark}

{The next lemma addresses the geometry of `bridge-shore' subdomains~$\Omega\subset\cD$ in more detail. In words, it says that the `ramps' always stay sufficiently far from the other parts of $\partial\Omega$ so that one can apply Proposition~\ref{prop:upper-bound} to them.}

\begin{lemma} \label{lem:ramps-not-screened} Let $\rho\le \frac13(1-\kappa^2)^{1/2}$ and~$\Omega\subset\cD$ be a `bridge-shore' subdomain constructed starting from points~$a,b\in\cD$. Then, for each~$z\in(aa_0)$ one has
\[
\dist(z,\partial\Omega\smallsetminus (aa_0))\,\ge\,\mathrm{cst}\cdot \min\{\|z-a\|_{2,1},\|z-a_0\|_{2,1}\},
\]
where~$\mathrm{cst}>0$ is a constant that depends on~$c_0$ only.
\end{lemma}
\begin{proof}
Let~$z\in (aa_0)$ and~$\phi$ be the angle between the directions of the `ramp'~$(aa_0)$ and the `bridge'~$(ab)$. 
Due to the upper bound in Lemma~\ref{lem:polar-coordinates} we have
\[
A\phi \|a_0-a\|_{2,1}\,\ge\, \dist(a_0,(ab))\,\ge\, {\tfrac12\rho}\|b-a\|_{2,1}\,,
\]
where~$A>0$ depends on~$c_0$ only. This gives the lower bound~$\phi\ge {\frac12 A^{-1}}$ as $\|a_0-a\|_{2,1}\le \dist(a,\partial\cD)\le \rho\|b-a\|_{2,1}$. It now follows from Corollary~\ref{lem:polar-coordinates} that
\[
\dist(z,(ab))\,\ge\,{\tfrac12 cA^{-1}}\cdot \|z-a\|_{2,1},
\]
where $c>0$ also depends on~$c_0$ only.

Since $\|a-a_0\|_{2,1}=\|a-z\|_{2,1}+\|z-a_0\|_{2,1}$ and~$a_0$ is the~$\|\cdot\|_{2,1}$-closest to~$a$ point of~$\partial\cD$, it is also easy to see from~\eqref{eq:dist-2,1-comparable} and the `distorted triangle inequality'~\eqref{eq:distorted-triangle-ineq} that
\[
\dist(z,\partial\cD)\,\ge\,\frac{1-\kappa^2}{1+\kappa^2}\|z-a_0\|_{2,1}\,.
\]
Together with~\eqref{eq:ramps-are-far} this concludes the proof.
\end{proof}

The next proposition justifies the relevance of `bridge-shore' subdomains for our adaptation of the renormalization argument from~\cite{duminil-manolescu-tassion-fractal}. Its proof is similar to the proof of \cite[Lemma~5.3]{duminil-manolescu-tassion-fractal} and relies upon the discrete holomorphicity of a fermionic observable, the lower bound provided by Proposition~\ref{prop:lower-bound} applied to the \emph{wired} `bridge' part of~$\partial\Omega$ (see Remark~\ref{rem:lower-bound} above), and the upper bound provided by Proposition~\ref{prop:upper-bound} that can be used for two `ramps' $(aa_0)$ and~$(b_0b)$ due to Lemma~\ref{lem:ramps-not-screened}.

\begin{proposition}\label{prop:lower-bound-w/o-legs} Provided that~$\rho$ is chosen small enough (depending on~$c_0$ only) the following holds: if~$\Omega$ is a `bridge-shore' subdomain of~$\cD$ whose `bridge' part~$(ba)^*$ of the boundary has {$\|\cdot\|_{2,1}$-}length~$L$, then
\[
\sum_{z\in(a_0b_0)_{\Omega}}\mathbb{P}^{0/1}_\Omega[v^\bullet(z)\leftrightarrow{(ba)^*}]\,\ge\,\mathrm{cst}\cdot L^{\frac12},
\]
where $\mathrm{cst}>0$ is a constant that depends on~$c_0$ only.
\end{proposition}
\begin{proof}
See Section~\ref{subsec:lower-bound-w/o-legs}.
\end{proof}
From now onwards we assume that the small parameter~$\rho=\rho(c_0)>0$ is fixed once forever so that Proposition~\ref{prop:lower-bound-w/o-legs} holds. In particular, we will be simply saying `bridge' instead of $\rho$-`bridge' for the wired part of the boundary of `bridge-shore' subdomains discussed in Definition~\ref{def:bridge}. 

\section{Adaptation of Duminil-Copin--Manolescu--Tassion's renormalization argument}
\label{sec:renormalizaion}
\setcounter{equation}{0}
In this section we adapt the renormalization argument from~\cite{duminil-manolescu-tassion-fractal} to our context. This allows us to derive our main result -- Theorem~\ref{thm:main-theorem} -- from {Proposition~\ref{prop:upper-bound} and Proposition~\ref{prop:lower-bound-w/o-legs}.} As compared to~\cite{duminil-manolescu-tassion-fractal}, we use `bridge-shore' subdomains introduced in the previous section instead of `corner' domains used in~\cite[Section~5.3]{duminil-manolescu-tassion-fractal}. Another detail is that~\cite[Definition~5.4]{duminil-manolescu-tassion-fractal} involves a technical trick that relies upon symmetries of the square lattice. 
Thus, we need to replace this part of the proof by a different argument, {which is contained in Section~\ref{subsec:first-moment} below.}
For readers familiar with~\cite{duminil-manolescu-tassion-fractal} this can be loosely summarized as follows: instead of finding one `bridge-shore' subdomain on the next scale, we find a sufficiently big number thereof and argue that at least in one of them the contribution of the `far' part of the boundary (similar to~$S_3$ in~\cite[Section~5.3]{duminil-manolescu-tassion-fractal}) can be neglected due to Proposition~\ref{prop:upper-bound} \emph{and} Theorem~\ref{thm:RSW-Mahfouf}.

\subsection{Existence of {sufficiently} many {disjoint} `bridge-shore' subdomains}
\label{subsec:fjords}

Recall that we chose a small constant~$\rho=\rho(c_0)>0$ in Section~\ref{subsec:bridge-shore} so that Proposition~\ref{prop:lower-bound-w/o-legs} holds. (Clearly, one can decrease $\rho$ further if needed without breaking the proof of this proposition.) We start with an auxiliary geometric lemma that allows us to find a `bridge' in a vicinity of a given point in~$\cD$ whose $\|\cdot\|_{2,1}$-length is uniformly comparable to the distance from this point to~$\partial\cD$. {This statement is intuitively clear; however, its proof requires some care as we need to work with special cuts that may have complicated geometry rather than with straight lines. Our main tool to handle this is Lemma~\ref{lem:polar-coordinates} given in Appendix.}

\begin{figure}
\centering \includegraphics[clip, trim=4.2cm 18.2cm 6cm 4cm, width=0.9\textwidth]{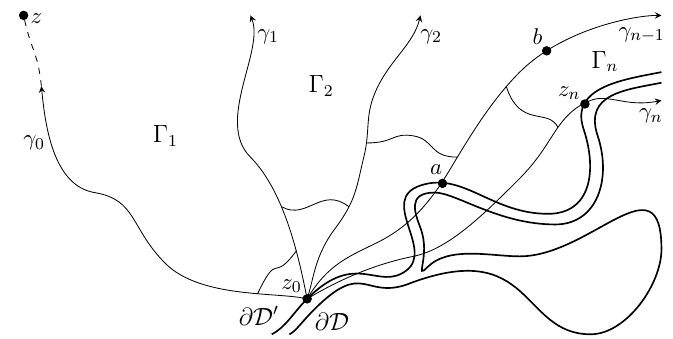}
\caption{\label{fig:bridge-near-pt}{Notation used in the proof of Lemma~\ref{lem:bridge-near-pt}.}}
\end{figure}

\begin{lemma} 
\label{lem:bridge-near-pt}
Provided that~$\rho>0$ is small enough, there exists a (big) constant $C=C(\rho,c_0)>0$ such that for each simply connected discrete domain~$\cD$ and each point~$z$ lying at the boundary of the main connected component of the $r$-interior of~$\Omega$ with $ r\ge C^2$ there exists a `bridge' in~$\cD$ {of a~$\|\cdot\|_{2,1}$-length at least $C^{-1}r$ that is contained in the Euclidean ball~$B(z,2r)$ together with its two `ramps'.}
\end{lemma}

\begin{proof} 
Let $A=A(c_0)$ be the constant from {Lemma~\ref{lem:polar-coordinates},} $N:=\lfloor 2\pi\exp(2\rho^{-1}A)\rfloor$, and $\ell:=\frac12(1-\kappa^2)^{1/2}r$. We will prove that one can find a `bridge' in~$\cD$ of $\|\cdot\|_{2,1}$-length at least~$\ell/N$ that is contained in~$B(z,2r)$ together with its two `ramps', which gives the desired statement with~$C=2(1-\kappa^2)^{-1/2}N$.

Denote~$r':={ \frac12\rho}\ell/N$ and let~$\cD'$ be the main connected component of the Euclidean $r'$-interior of~$\cD$. Further, let~$z_0\in\partial\cD'$ be the $\|\cdot\|_{2,1}$-closest point to~$z$ at the boundary of~$\cD'$ and~$\gamma_0$ be the initial part of $\|\cdot\|_{2,1}$-length $\ell$ of the special cut that starts at~$z_0$ and passes through~$z$. It is easy to see that $\gamma_0\smallsetminus\{z_0\}\subset\cD'$ since
\[
\|z-z_0\|_{2,1}\,\ge\, (1-\kappa^2)^{1/2}\dist(z,z_0)\,\ge\, (1-\kappa^2)^{1/2}(r-r')\,\ge\,\ell.
\]

Without loss of generality, assume that~$\gamma_0$ goes in the negative direction and denote by~$\gamma_n$ the special cuts of length~$\ell$ emanating from~$z_0$ and having directions $\exp\big(i(\pi-\frac\rho{2A}\log (n+1))\big)$; {see Fig.~\ref{fig:bridge-near-pt} for the notation.} {Let
\[
\Gamma_n\,=\,\left\{z'\in\cD': \begin{array}{l} \tfrac{n}{N}\ell\le\|z'-z_0\|_{2,1}\,\le\, \ell\ \text{and the special cut going from~$z_0$ to~$z'$}\\
\text{has direction~$\exp[i(\pi-\phi)]$ with~$\tfrac{\rho}{2A}\log n<\phi\le \tfrac{\rho}{2A}\log(n+1)$.}\end{array}\right\}
\]
Since~$\cD'$ is a simply connected domain and~$z_0\in\cD$, we should have~$\Gamma_n\cap \partial\cD'\ne\varnothing$ at least for one index~$n=1,\ldots,N$. Let $n$ be the \emph{minimal} index for which there exists a boundary point~$z_n\in\Gamma_n\cap \partial\cD'$.}

Recall that $(\gamma_0\smallsetminus z_0)\cap\partial\cD'=\varnothing$ and denote by~$a\in\gamma_{n-1}$ the last (that is, $\|\cdot\|_{2,1}$-furthest from~$z_0$) intersection point of~$(z_0 z_{n-1})\subset \gamma_{n-1}$ and~$\partial\cD'$. (We set~$a=z_0$ if no such intersection exists.) { Since~$\Gamma_{n-1}\cap\partial\cD'=\emptyset$ we have
$\|a-z_0\|_{2,1}\le (n-1)\ell/N$.}
Let~$b\in\gamma_{n-1}$ be chosen so that~$\|b-z_0\|_{2,1}=\|z_n-z_0\|_{2,1}$. We have
\[
\|b-a\|_{2,1}\,=\,\|b-z_0\|_{2,1}-\|a-z_0\|_{2,1}\,\ge\,{\frac {n\ell}N -\frac{(n-1)\ell}N}\,=\,\frac{\ell}{N}\,.
\]

By construction, the special cut $(ab)\subset\gamma_{n-1}$ lies inside~$\cD'$ and hence inside the ${ \frac12\rho}\|b-a\|_{2,1}$-interior of~$\cD$ since ${ \frac12\rho}\|b-a\|_{2,1}\ge {\frac12\rho}\ell/N=r'$. Also, we have $\dist(a,\partial\cD)=r'$ and it follows from {Lemma~\ref{lem:polar-coordinates}} that
\begin{align*}
\dist(b,\partial\cD)\,\le\,\dist(b,z_n)+r'\,&\le\, A\cdot \frac{\rho(\log (n+1)-\log n)}{2A}\,\|b-z_0\|_{2,1}+r'\\
&\le\, \frac{\rho}{2n}\biggl(\|b-a\|_{2,1}+\frac{(n-1)\ell}{N}\biggr)+r'\\
&\le\,\frac{\rho}{2}\,\|b-a\|_{2,1}+r'\,\le\,\rho\|b-a\|_{2,1}.
\end{align*}

Thus, the special cut~$(ab)$ satisfies the quantitative conditions listed in {Definition~\ref{def:bridge}.} Also, one has $(ab)\subset B(z_0,\frac12r)$ since
\[
\diam(\gamma_{n-1})\le (1-\kappa^2)^{-1/2}\ell=\tfrac12r\,.
\]
Moreover, if $a_0,b_0\in\partial\cD$ are the~$\|\cdot\|_{2,1}$-closest boundary points to~$a$ and~$b$, then
\[
\diam((b_0b))\le (1-\kappa^2)^{-1/2}\|b-b_0\|_{2,1}\le (1-\kappa^2)^{-1/2}\dist(b,\partial\cD)\le\tfrac12\rho r\,.
\]
and similarly for~$(aa_0)$. Hence,
\[
\diam\big({ (b_0b)\cup (ba)^*\cup (aa_0)}\big)\le (\tfrac12+\rho)r\,.
\]
In particular, this cut cannot separate~$z$ and the origin in~$\cD$ as~$z$ belongs to the main connected component of the~$r$-interior of~$\Omega$. By construction,~$z$ {can be} connected in~$\cD'$ to the \emph{left} side of~$(ab)$ {by a curve going through the domains~$\Gamma_1,\ldots,\Gamma_{n-1}$. Hence,} the `bridge-shore' subdomain of~$\cD$ cut by the curve~$ (b_0b)\cup(ba)^*\cup(aa_0)$ does not contain~$z$ and the origin. Since~$\dist(z,z_0)\le r$, it is also clear that this cut lies in the ball~$B(z_0,(\frac12+\rho)r)\subset B(z,2r)$. 
\end{proof}

We now use Lemma~\ref{lem:bridge-near-pt} to prove the existence of arbitrarily many disjoint `bridge-shore' subdomains $\Omega_1,\ldots,\Omega_N$ of~$\cD$ such that the corresponding `bridges' lie in the disc~$B_{2R}$ and have $\|\cdot\|_{2,1}$-lengths or order~$R/N$.

\begin{proposition} \label{prop:many-bridges}
Let $N\in\mathbb{N}$ and $\cD$ be a simply connected discrete domain such that~$0\in\cD$. Provided that $R:=\dist(0,\partial\cD)$ is sufficiently large (depending on~$N$ only) one can find~$N$ `bridges' $\gamma_1,\ldots,\gamma_N$ that lie in the disc $B_{2R}$, have {$\|\cdot\|_{2,1}$-lengths at least $R/4CN$} where $C$ is the constant from Lemma~\ref{lem:bridge-near-pt}, and such that the corresponding `bridge-shore' subdomains $\Omega_1,\ldots,\Omega_N$ of~$\cD$ are disjoint.
\end{proposition}
\begin{proof} Let~$ r=R/4N$ and $\cD^{(r)}$ be the main connected component of the $r$-interior of~$\cD$. Introduce auxiliary variables~$k_1,\ldots,k_N\in\mathbb{N}_0$ {whose meaning in the forthcoming proof is `the number of tries used to cut~$\Omega_j$'}.
To start, we {set~$n:=1$ and $k_1:=1$,} choose an arbitrarily point~$z_{1,1}\in\partial\cD^{(r)}\cap B_{ 2(R-r)}$, and use Lemma~\ref{lem:bridge-near-pt} to construct the corresponding `bridge-shore' subdomain \mbox{$\Omega_{1,1}\subset \cD$} whose `bridge' {and two `ramps'} belong to the disc~$B_{ 2r}(z_{1,1})\subset B_{2R}$. 

We now inductively construct points~$z_{n,k}$ and the corresponding domains as follows. 
Introduce the notation
\[
\textstyle \cD_n:=\cD\smallsetminus\bigcup_{j=1}^n\Omega_{j,k_j}
\]
and let~$\cD_n^{(r)}$ be the main connected component of the $r$-interior of~$\cD_n$. On each next step we choose an arbitrarily point~$z_*\in \partial\cD^{(r)}_n\cap B_{ 2(R-r)}$ that is at least \mbox{$ 4r$-away} from the previously constructed points~$z_{1,k_1},\ldots,z_{n,k_n}$ and construct the corresponding `bridge-shore' subdomain~$\Omega_*\subset\cD$. Then, provided that~$n<N$,
\begin{itemize}
\item if $\Omega_*\cap \Omega_{j,k_j}=\varnothing$ for all~$j\le n$, then we declare $z_{n+1,1}:=z_*$, $\Omega_{n+1,1}:=\Omega_*$, { set~$k_{n+1}:=1$, and increase~$n$ by one;}
\item if $\Omega_*\supset\Omega_{j,k_j}$ for some $j\le n$ (that is necessarily unique), then we declare $z_{j,k_j+1}:=z_*$, $\Omega_{j,k_j+1}:=\Omega_*$ and increase~$k_j$ by one.
\end{itemize}

First, note that $N-1$ discs of radius~$ 2r=R/2N$ centered at the previously constructed points~$z_{j,k_j}$ cannot cover 
{ the set~$\partial\cD_n^{(r)}\cap B_{ 2(R-r)}$ as this part of the boundary of~$\cD_n^{(r)}$ always contains a continuum of diameter~$R=\dist(0,\partial\cD)$.}
Thus, we can always find the next point~$z_*$. Second, note that the `ramp' parts of~$\partial\Omega_*$ are always attached to the boundary~$\partial\cD$ of the original domain rather than to the previously constructed `bridges' and their `ramps' since we choose~$z_*$ sufficiently far from the previously constructed points~$z_{j,{k_j}}$. In particular, this implies that either $\Omega_*\cap\Omega_{j,k_j}=\varnothing$ or $\Omega_*\supset \Omega_{j,k_j}$ for all~$j\le n$. Third, note that we cannot increase any~$k_j$ {infinitely many times}: this would lead to a construction of an infinite sequence $\Omega_{j,1}\subset\Omega_{j,2}\subset\ldots$ of nested `bridge-shore' domains, which is impossible since we require that each new `bridge' lies in the ${ \frac12\rho} r$-interior of the previously constructed domain. This implies that the ${ \frac14\rho} r$-vicinities of all `bridges' do not intersect and hence one cannot find infinitely many of those having lengths bounded from below in a finite volume. Thus, the aforementioned inductive construction always reaches~$n=N$ and we can declare~$\Omega_j:=\Omega_{j,k_j}$ for {all $j$} when it stops.
\end{proof}

\subsection{First moment estimate}
\label{subsec:first-moment}
{ In this section we use Proposition~\ref{prop:lower-bound-w/o-legs}, Proposition~\ref{prop:many-bridges}, and Theorem~\ref{thm:RSW-Mahfouf} to prove that the expected number of boundary points of~$\cD$ that are connected to~$B_{\frac12R}$ in the ball $B_{5R}$ has the right order of magnitude; see Corollary~\ref{cor:lower-bound-to-bdry} below. Let us repeat that our argument is \emph{not} fully similar to~\cite[Section~5.3]{duminil-manolescu-tassion-fractal} as we cannot rely upon the symmetries of the square lattice that were used in~\cite{duminil-manolescu-tassion-fractal} in order to construct `corridor' domains and to neglect the contribution of the boundary of~$\cD$ that lies outside the disc~$B_{5R}$.}

\begin{proposition} \label{prop:lower-bound-w/o-far} There exists~$N_0=N_0(c_0)\in\mathbb{N}$ such that the following holds. For each simply connected discrete domain~$\cD$ with~$0\in\cD$ and~$R=\dist(0,\partial\cD)$ large enough (depending on~$c_0$ only) one can find at least one `bridge-shore' subdomain $\Omega=\Omega_j$ of~$\cD$ among those constructed in Proposition~\ref{prop:many-bridges} such that
\[
\sum_{z\in (a_0b_0)_{ \Omega}\cap B_{5R}}\mathbb{P}^{0/1}_{\Omega}\big[v^\bullet(z)\xleftrightarrow{B_{5R}\,}{ (ba)^*}\big]\,\ge\, \mathrm{cst} \cdot R^{\frac12},
\]
where $\mathrm{cst}>0$ is a constant that depends on~$c_0$ only.
\end{proposition}

\begin{proof} Choose a sufficiently large~$M\in\mathbb{N}$ and let $a_1,a_2,\ldots,a_M$ be equidistributed points on the circle of radius~$\frac92R$ centered at the origin. Denote by~$\gamma_k$ the special cut going from~$a_k$ to~$a_{k+1}$ (we set~$a_{M+1}:=a_1$) and note that~$\diam(\gamma_k)=O(M^{-1}R)$; see Section~\ref{subsec:cuts} for a discussion. These special cuts may have complicated local geometry and, in particular, nearby~$\gamma_k$ can intersect each other so that the closed curve~$\bigcup_{k=1}^M \gamma_k$ has many local self-intersections. Denote by~$\gamma'$ a simple loop obtained by the chronological erasure (say, starting from~$a_0$) of microscopic loops in this closed curve. (We stop the procedure when a loop surrounding the origin is found and discard the initial part of the curve that goes from~$a_0$ to the hitting point.) It is easy to see that $\gamma'=\bigcup_{k=1}^M\gamma_k'$, where each $\gamma_k'$ is a \emph{subsegment} of~$\gamma_k$. {Indeed, erasing a loop from a curve cannot create an appearance of parts of the same~$\gamma_k$ at two difference places, so this property is preserved at every step.} (It may happen that~$\gamma_k'=\varnothing$ for some~$k$ if~$\gamma_k$ was completely erased during the loop-erasure procedure described above.)

Denote by~$B'_{5R}$ the discrete domain bounded by~$\gamma'$ and let~$\Omega'$ be the main (that is, adjacent to the {~$(ba)_\Omega=(ba)^*$}) connected component of~$\Omega\cap B'_{5R}$, where~$\Omega=\Omega_j$ is one of the `bridge-shore' subdomains of~$\cD$ constructed in Proposition~\ref{prop:many-bridges}. Note that~$\Omega'$ is a `bridge-shore' subdomain of~$\cD'$. Using the monotonicity with respect to boundary conditions and Proposition~\ref{prop:lower-bound-w/o-legs} applied in ~$\Omega'$ we obtain a lower bound
\begin{align}
\sum_{z\in(a_0b_0)_{ \Omega}}\mathbb{P}_\Omega^{0/1}\big[v^\bullet(z)\xleftrightarrow{B_{5R}\,}{ (ba)^*}\big]\,&\ge\!\!\!
\sum_{z\in(a_0b_0)_{ \Omega}\cap B'_{5R}}\!\!\!\mathbb{P}_{\Omega'}^{0/1}\big[v^\bullet(z)\leftrightarrow {(ba)^*}\big]
\notag\\
&\ge\,\mathrm{cst}\cdot R^{\frac12}\ -\!\!\!\!\!\!\sum_{z\in \partial B'_{5R}\cap\partial\Omega'}\!\!\!\mathbb{P}_{\Omega'}^{0/1}\big[v^\bullet(z)\leftrightarrow{ (ba)^*}\big].
\label{eq:x-lower-w/o-far-i}
\end{align}
Thus, it is sufficient to estimate the last sum from above. 

Recall that the boundary of~$B'_{5R}$ is composed of special cuts. Let~$\Lambda'_k$ be a Dobrushin domain of diameter~$O(M^{-1}R)$ with the free arc~$\gamma'_k$ to which one can apply Proposition~\ref{prop:upper-bound}. (E.g., consider a domain bounded by~$\gamma'_k$, two straight segments of length~$M^{-1}R$ perpendicular to the direction of~$\gamma'_k$ and its translation in this direction.) We emphasize that we do \emph{not} require that~$\Lambda'_k\subset B'_{5R}$; however, it is clear that $\Lambda'_k\cap B_{4R}=\varnothing$ provided that~$M$ is chosen large enough. Let~$\lambda'_k$ be the wired arc of~$\Lambda'_k$ and~$\operatorname{Cross}_{2R,4R}(\Omega)$ be the event that~$\Omega\cap (B_{4R}\smallsetminus B_{2R})$ is crossed in the radial direction. Due to the monotonicity with respect to boundary conditions and Proposition~\ref{prop:upper-bound} we have the upper bound
\begin{align}
\sum_{z\in \partial B'_{5R}\cap\partial\Omega'}\mathbb{P}_{\Omega'}^{0/1}[v^\bullet(z)\leftrightarrow {(ba)^*}]
&\,\le\,\mathbb{P}_{\Omega'}^{0/1}\big[\operatorname{Cross}_{2R,4R}(\Omega)\big]\cdot\sum_{k=1}^M\sum_{z\in\gamma'_k}\mathbb{P}_{\Lambda'_k}^{0/1}[v^\bullet(z)\leftrightarrow \lambda'_k]
\notag \\
&\le\, \mathbb{P}_{\Omega'}^{0/1}\big[\operatorname{Cross}_{2R,4R}(\Omega)\big]\cdot O\big(M^{\frac12}R^{\frac12}\big)
\label{eq:x-lower-w/o-far-ii}
\end{align}
as we have~$M$ terms, each of which is~$O((M^{-1}R)^{1/2})$.

It follows from Theorem~\ref{thm:RSW-Mahfouf} for the dual model and the FKG inequality that the probability that \emph{none} of the events~$\operatorname{Cross}_{2R,4R}(\Omega_j)$ happens is at least~$p_\mathrm{circ}$. Since the domains~$\Omega_j$ are disjoint, this can be written as
\[
\textstyle \prod_{j=1}^N\big(1-\mathbb{P}_{\Omega'_j}^{0/1}\big[\operatorname{Cross}_{2R,4R}(\Omega)\big]\big)\,\ge\,p_\mathrm{circ}\,.
\]
Hence, we must have
\begin{equation}
\label{eq:x-lower-w/o-far-iii}
\mathbb{P}_{\Omega'}^{0/1}\big[\operatorname{Cross}_{2R,4R}(\Omega)]\,\le\,1-p_\mathrm{circ}^{\frac1N}
\end{equation}
for at least one of the `bridge-shore' domains~$\Omega=\Omega_j$, $j=1,\ldots,N$. Since the choice of~$M$ made above does not depend on~$N$, we can now combine the estimates \eqref{eq:x-lower-w/o-far-i}--\eqref{eq:x-lower-w/o-far-iii} and obtain the desired lower bound provided that~$N\gg |\log p_\mathrm{circ}|M^{1/2}$.
\end{proof}

\begin{corollary} \label{cor:lower-bound-to-bdry} The following estimate holds for all simply connected discrete domains~$\cD$ with~$0\in\cD$ and $R=\dist(0,\partial\cD)$:
\[
\sum_{z\in\partial\cD}\mathbb{P}_\cD^0\big[v^\bullet(z)\xleftrightarrow{B_{5R}\,}B_{\frac12R}\big]\ge\mathrm{cst}\cdot R^{\frac12},
\]
where $\mathrm{cst}>0$ is a constant that depends on~$c_0$ only.
\end{corollary}
\begin{proof} 
Let~$\Omega=\Omega_k$ be the `bridge-shore' subdomain of~$\cD$ from Proposition~\ref{prop:lower-bound-w/o-far}. Denote by~$\gamma={ (ba)^*}$ the `bridge' part of its boundary. {Recall that according to Proposition~\ref{prop:many-bridges} and} the definition of `bridge-shore' subdomains (see Section~\ref{subsec:bridge-shore}) $\gamma$ lies in the ${ \frac12\rho R/4CN_0}$-interior of~$\cD$. 
{ Let~$\mathrm{Circ}({ \cD\smallsetminus \gamma})$ be the event that there is an open circuit in the annulus~$\cD\smallsetminus\gamma$.}
A standard application of the `usual' RSW estimates from Theorem~\ref{thm:RSW-Mahfouf} gives the lower bound
\[
\mathbb{P}^0_\cD\big[\mathrm{Circ}({ \cD\smallsetminus \gamma})\big]\,\ge\,\mathrm{cst}>0,
\]
where $\mathrm{cst}>0$ is a constant that depends on~$c_0$ only since~$\rho$, $C$, and~$N_0$ are already fixed. Moreover, {this circuit surrounding~$\gamma$} is also connected to~$B_{\frac12R}$ with positive probability that is uniformly bounded from below. It now follows from the FKG inequality and the monotonicity with respect to boundary conditions that
\[
\sum_{z\in\partial\cD}\mathbb{P}_\cD^0\big[v^\bullet(z)\xleftrightarrow{B_{5R}\,}B_{\frac12R}\big]\,\ge\,\mathrm{cst}\cdot\!\!\!\! \sum_{z\in (a_0b_0)}\mathbb{P}_\Omega^{0/1}\big[v^\bullet(z)\xleftrightarrow{B_{5R}\,} (ba)\big]
\]
and Proposition~\ref{prop:lower-bound-w/o-far} gives the lower bound~$\mathrm{cst}\cdot R^{1/2}$ for the last sum.
\end{proof}

\subsection{Branching process and the proof of Theorem~\ref{thm:main-theorem}}
\label{subsec:proof-main-thm}
We are now ready to prove our main result.
Given $R>0$ denote by $\widehat{\cD}_R$ the set of all simply connected discrete domains~$\cD$ such that~$0\in\Omega$ and $R_\Omega:=\dist(0,\partial\cD)\le R$. Let
\[
p(R):=\inf_{\cD\in\widehat{\cD}_R}\mathbb{P}^0_\Omega[\partial\cD\xleftrightarrow{B_{5R_\Omega}\,}B_{\frac12R_\Omega}].
\]
{ By definition, $p(R)$ is a non-increasing function of~$R$ and our goal is to prove that~$\lim_{R\to+\infty}p(R)>0$.} Due to the finite energy property of the model this probability is uniformly bounded from below if~$R_\Omega$ is bounded {from above, which means that~$p(R_0)>0$ for each fixed~$R_0$.} Therefore, in order to prove Theorem~\ref{thm:main-theorem} it is sufficient to show that there exist big $m_0\in\mathbb{N}$ and small $p'_*>0$ such that
\begin{equation}
\label{eq:induction-step}
p(R)\,\ge\,\min\{p(2^{-2m_0-2}R),p'_*\}\ \ \text{for all}\ \ R\ge 2^{3m_0}.
\end{equation}
(One then sets~$p_*:=\min\{p'_*,p(2^{3m_0})\}$ and proceeds by induction over scales.) The following proof is an adaption of the argument given in~\cite[Section~5.3]{duminil-manolescu-tassion-fractal} to our setup.

\begin{proof}[Proof of the estimate~\eqref{eq:induction-step}] 
Choose a large~$m\in\mathbb{N}$ and let~$\cD_m\subset\cD$ be an inner approximation of the main connected component of~$\cD\cap B_{6R}$ on the \emph{square grid} with mesh size~$2^{-2m}R$ such that the distance from each boundary point of~$\cD_m$ to~$\partial\cD$ is in~$[2^{-2m}R;\,2\sqrt{2}\cdot 2^{-2m}R]$. (To construct~$\cD_m$, consider the union of all squares such that their eight neighbors lie inside~$\cD\cap B_{6R}$.)
Further, subdivide each boundary segment of $\cD_m$ of length~$2^{-2m}R$ into~$2^m$ subsegments of length~$2^{-3m}R$ and denote by~$a_1,\ldots,a_M$ the endpoints of these segments so that~$\partial\cD_m$ is the closed polyline $(a_1a_2\ldots a_{M+1})$, where we set $a_{M+1}=a_1$ as usual. Note that we have the uniform estimate \mbox{$M\le 2^m\cdot O((2^{2m})^2)=O(2^{5m})$} as we consider only squares inside~$B_{6R}$.

Similarly to the proof of Proposition~\ref{prop:lower-bound-w/o-far}, let~$\gamma_k$ be the special cut going from~$a_k$ to~$a_{k+1}$ and $\gamma'=\bigcup_{k=1}^M\gamma'_k$ be a simple closed curve obtained by a chronological erasure of microscopic loops from the concatenation of all~$\gamma_k$. (Each~$\gamma'_k$ is a \emph{subsegment}, possibly empty, of~$\gamma_k$.) Let~$\cD'$ be the discrete domain bounded by~$\gamma'$. Due to Corollary~\ref{cor:lower-bound-to-bdry} applied in~$\cD'$ we have the lower bound
\[
\sum_{k=1}^M\sum_{z\in\gamma_k'}\mathbb{P}_{\cD'}^0\big[v^\bullet(z)\xleftrightarrow{B_{5R'}\,}B_{\frac12R'}\big]\,\ge\,\mathrm{cst}\cdot R^{\frac12},
\]
where~$R'=\dist(0,\partial\cD')$; note that $R'\ge R-O(2^{-2m}R)$. 

As in the proof of Proposition~\ref{prop:lower-bound-w/o-legs}, let~$\Lambda'_k$ be a Dobrushin domain of diameter~$O(2^{-3m}R)$ with the free arc~$\gamma'_k$ to which one can apply Proposition~\ref{prop:upper-bound}. (E.g., consider the domain bounded by~$\gamma'_k$, two straight segments of length~$2^{-3m}R$ perpendicular to the direction of~$\gamma'_k$ and its translation in this direction.) We do \emph{not} require that~$\Lambda'_k\subset \cD'$; however, it is clear that $\Lambda'_k\subset\cD$ provided that~$m$ is chosen large enough. Denote by~$\lambda'_k:=\partial\Lambda'_k\smallsetminus\gamma'_k$ the wired arc of~$\Lambda'_k$. Due to the monotonicity with respect to boundary conditions, for each~$z\in\gamma'_k$ we have the estimate
\begin{align*}
\mathbb{P}_{\cD'}^0\big[v^\bullet(z)\xleftrightarrow{B_{5R'}\,}B_{\frac12R'}\big] \,&\le\, \mathbb{P}^{0/1}_{\Lambda'_k}[v^\bullet(z)\leftrightarrow \lambda'_k]\cdot \mathbb{P}_{\cD'}^0[\Lambda'_k\cap\cD'\xleftrightarrow{B_{5R'}\,}B_{\frac12R'}]\,.
\end{align*}
Trivially, $\mathbb{P}_{\cD'}^0[\Lambda'_k\cap\cD'\xleftrightarrow{B_{5R'}\,}B_{\frac12R'}]\le \mathbb{P}_{\cD}^0[\Lambda'_k\xleftrightarrow{B_{5R'}\,}B_{\frac12R'}]$. Therefore, the summation over~$z\in\gamma'_k$ and Proposition~\ref{prop:upper-bound} applied in~$\Lambda'_k$ give the upper bound
\[
\sum_{k=1}^M\sum_{z\in\gamma_k'}\mathbb{P}_{\cD'}^0\big[v^\bullet(z)\xleftrightarrow{B_{5R'}\,}B_{\frac12R'}\big]\,\le\, O\big((2^{-3m}R)^{\frac12}\big)\cdot\sum_{k=1}^M \mathbb{P}_{\cD}^0[\Lambda'_k\xleftrightarrow{B_{5R'}\,}B_{\frac12R'}]\,.
\]
Hence,
\[
\sum_{k=1}^M \mathbb{P}_{\cD}^0[\Lambda'_k\xleftrightarrow{B_{5R'}\,}B_{\frac12R'}]\,\ge\,\mathrm{cst}\cdot 2^{\frac32m},
\]
where $\mathrm{cst}>0$ depends on~$c_0$ only and, in particular, does not depend on~$m$.

Following~\cite{duminil-manolescu-tassion-fractal} we say that the domain $\Lambda'_k$ is \emph{activated} if it is connected to~$B_{\frac12R'}$ inside~$B_{5R'}$. Let
\[
\xi:=\#\{1\le k\le M:\ \Lambda'_k\ \text{is activated}\},
\]
one can then write the last lower bound as
\begin{equation}
\label{eq:expectation-xi}
\mathbb{E}^0_\cD[\,\xi\,]\,\ge\,\mathrm{cst}\cdot 2^{\frac32m}.
\end{equation}
Now let~$z_k$ be the point lying on the shortest straight segment connecting~$\gamma'_k$ and~$\partial\cD$ such that
\[
r_k:=\dist(z_k,\partial\cD)=\tfrac1{12}\dist(\gamma'_k,\partial\cD).
\]
If~$m$ is large enough, then we have $r_k\le \frac14\cdot 2^{-2m}R$ and~$R-R'\ge\frac78\cdot 2^{-2m}R$. In particular, this implies that
\[
5R'+\dist(\gamma'_k,z_k)+5r_k\,=\,5R'+16r_k\,\le\, 5R\,.
\]
Therefore, the FKG inequality implies that for each~$k=1,\ldots,M$ one has
\begin{align}
\mathbb{P}_\cD^0\big[B_{\frac12R}\xleftrightarrow{B_{5R}\,}\partial\cD\big]\,&\ge\,\mathbb{P}_\cD^0\big[\Lambda'_k\ \text{is activated}\big]\cdot p_\mathrm{circ}^3\cdot \mathbb{P}_\cD^0\big[B_{\frac12r_k}(z_k)\xleftrightarrow{B_{5r_k}(z_k)\,}\partial\cD\big] \notag \\
&\ge\,\mathbb{P}_\cD^0\big[\Lambda'_k\ \text{is activated}\big]\cdot p_\mathrm{circ}^3\cdot p(2^{-2m-2}R)\,. 
\label{eq:x-one-activated}
\end{align}
{ The factor~$p_\mathrm{circ}^3$ comes from requiring the existence of a circuit in the annulus $B_{r_k}(z_k)\smallsetminus B_{\frac12r_k}(z_k)$ and two additional circuits centered at points lying at distances at distances $\frac14\dist(\gamma'_k,\partial\cD)$ and $\frac34\dist(\gamma'_k,\partial\cD)$ from~$\partial\cD$ on the shortest segment connecting~$\gamma'_k$ and~$\partial\cD$ that are needed to connect $B_{\frac12r_k}(z_k)$ and the crossing going from~$\Lambda'_k$ to~$B_{\frac12R}$.}

The key part of the argument is to benefit from the expected existence of \emph{many} activated~$\Lambda'_k$ that is suggested by the lower bound~\eqref{eq:expectation-xi}. To this end, note that there is a small universal constant~$\alpha>0$ such that the following holds:
\begin{center}
each subset~$\{z_{k_1},\ldots, z_{k_n}\}\subset\{z_1,\ldots,z_M\}$ contains a sub-subset consisting\\
of at least~$\lceil \alpha 2^{-m} n\rceil$ points lying at distance $\ge 3\cdot 2^{-2m}R$ from each other.
\end{center}
Indeed, each edge of the square grid with mesh size~$2^{-2m}R$ has $O(1)$ edges of the same grid at $O(2^{-2m}R)$ distance, so no~$z_k$ has more than~$O(2^m)$ other points at~$O(2^{-2m}R)$ distance from it, $2^m$ per each of the $O(1)$ aforementioned edges of the grid. (One can explicitly take~$\alpha=\frac{1}{100}$.) Let us now use the estimate~\eqref{eq:expectation-xi} and find~$m_0\in\mathbb{N}$ such that
\begin{equation}
\label{eq:expectation-xi-bis}
\mathbb{E}[\xi]\,\ge\,\alpha^{-1}2^{m_0+1}p_\mathrm{circ}^{-3}\,.
\end{equation}

Recall that we have a uniform upper bound~$\xi\le M\le M_0=O(2^{5m_0})$. {Suppose} first that
\[
p_\mathrm{circ}^3\cdot p(2^{-2m_0-2}R)\,\ge\, \alpha^{-1} 2^{m_0-1}M_0^{-1}. 
\]
In this case it easily follows from~\eqref{eq:x-one-activated} that
\begin{align}
\mathbb{P}_\cD^0\big[B_{\frac12R}\xleftrightarrow{B_{5R}\,}\partial\cD\big]\, &\ge\, \mathbb{P}_{\cD}^0\,[\text{$\Lambda'_k$ is activated for some $k$}\,]\cdot p_\mathrm{circ}^3\cdot p(2^{-2m_0-2}R)
\notag\\
& \ge\,\frac{\mathbb{E}_{\cD}^0[\xi]}{M_0}\cdot \alpha^{-1}2^{m_0-1}M_0^{-1}\,\ge\,p'_*>0
\label{eq:x-p(R)-bound-i}
\end{align}
for an appropriately chosen constant~$p'_*\ge M_0^{-2}$. 

Now {suppose} that
\[
p_\mathrm{circ}^3\cdot p(2^{-2m_0-2}R)\,\le\, \alpha^{-1} 2^{m_0-1}M_0^{-1},
\]
which is the key part of the argument. In this case we can benefit from the fact that the two discs $B_{5r_k}(z_k)$ and $B_{5r_\ell}(z_\ell)$ are disjoint if $\dist(z_k,z_\ell)\ge \frac52\cdot 2^{-2m}R$. Due to monotonicity with respect to boundary conditions, this allows us to improve the lower bound~\eqref{eq:x-one-activated} by considering several activated $\Lambda'_k$ that stay sufficiently far from each other. Namely, we have
\begin{align*}
\mathbb{P}_\cD^0\big[B_{\frac12R}\xleftrightarrow{B_{5R}\,}\partial\cD\big]\,&\ge\, \sum_{n=1}^{M_0}\mathbb{P}_\cD^0[\xi=n]\big(1-(1-p_\mathrm{circ}^3\cdot p(2^{-2m_0-2}R))^{\lceil \alpha 2^{-m_0}n\rceil}\big) \\
&\ge\,\sum_{n=1}^{M_0}\mathbb{P}_\cD^0[\xi=n]\cdot\tfrac12\lceil \alpha 2^{-m_0}n\rceil p_\mathrm{circ}^3\cdot p(2^{-2m_0-2}R) \\
&\ge\,\mathbb{E}_{\cD}^0[\xi]\cdot \alpha 2^{-m_0-1}p_\mathrm{circ}^3\cdot p(2^{-2m_0-2}R)\,\ge\,p(2^{-2m_0-2}R)\,,
\end{align*}
where we used~\eqref{eq:expectation-xi-bis} and the inequality~$(1-q)^n \le e^{-qn}\le 1-\frac12qn$ for~$qn\le 1$. Together with the estimate~\eqref{eq:x-p(R)-bound-i} this completes the induction step~\eqref{eq:induction-step} and the proof of Theorem~\ref{thm:main-theorem}.
\end{proof}

\section{Proofs of {Proposition~\ref{prop:upper-bound}, Proposition~\ref{prop:lower-bound}, and Proposition~\ref{prop:lower-bound-w/o-legs}}}
\label{sec:fermionic-proofs}
\setcounter{equation}{0}

This section relies upon a standard tool used for the analysis of the planar Ising model -- fermionic observables viewed as discrete holomorphic functions -- and uses the notation introduced in~\cite{chelkak-semb} in the s-embeddings context. The proof of Proposition~\ref{prop:lower-bound} given below goes by contradiction, which in particular requires us to work with \emph{rescaled} s-embeddings~$\cS^\delta$ and to study subsequential limits of fermionic observables defined on these rescaled embeddings as~$\delta\to 0$. In what follows we always assume that $\cS^\delta$ have the property~\UnifD, i.e., that the lengths of all edges of $\cS^\delta$ are uniformly comparable to~$\delta$ (rather than to~$1$) and all quads of~$\cS^\delta$ have angles that are uniformly bounded away from~$0$. 

\subsection{Fermionic observables: s-holomorphicity, closed differential forms, and positivity on special cuts} \label{subsec:ferm-obs}
Given a Dobrushin domain~$\Omega^\delta$ on~$\cS^\delta$ with free boundary conditions on~$(ab)_{\Omega^\delta}$ and wired boundary conditions on $(ba)_{ \Omega^\delta}$, the complexified fermionic observable is a function $F^\delta:\diamondsuit(\Omega^\delta)\to\C$ defined on quads (including boundary half-quads) that form $\Omega^\delta$. This function satisfies the \emph{s-holomorphicity} condition:
\begin{equation}
\label{eq:Pr=Pr}
\Pr{F^\delta(z)}{\eta_c\R}=\Pr{F^\delta(z')}{\eta_c\R},
\end{equation}
where
\begin{equation}
\label{eq:eta-def}
\eta_c=e^{i\frac\pi4}\exp\big[-\tfrac{i}2\arg(\cS^\delta(v^\bullet)-\cS^\delta(v^\circ))\big]
\end{equation}
and $z$ and $z'$ are the two quads adjacent to an edge $c=(v^\circ v^\bullet)$ of the s-embedding. The values~\eqref{eq:Pr=Pr} can be defined either using Kadanoff--Ceva fermions as it is done in~\cite{chelkak-semb} or, equivalently, by evaluating the expected value of the \emph{signed} probability that the FK-Ising interface from~$a^\delta$ to~$b^\delta$ crosses the edge~$(v^\circ v^\bullet)$ similarly to~\cite{smirnov-icm06,chelkak-smirnov-ising,duminil-parafermion} and dividing this value by~$(-i(\cS^\delta(v^\bullet)-\cS^\delta(v^\circ)))^{1/2}$. 
The fact that thus defined values on four edges of a quad $z$ can be written as the four projections of a single complex number~$F^\delta(z)$ onto directions~\eqref{eq:eta-def} follows from the so-called propagation equation for Kadanoff--Ceva fermions, see~\cite{chelkak-semb} for details. 

Let us now list (some of) the key properties of fermionic observables~$F^\delta$ that we need below. Again, we refer the interested reader to~\cite{chelkak-semb} for more details.

\begin{proposition} \label{prop:closed-forms} Let $F^\delta:\diamondsuit(\Omega^\delta)\to\C$ be an s-holomorphic function on~$\Omega^\delta$. Then,\\[2pt]
(i) One can consistently define a closed differential form $F^\delta d\cS^\delta+i\overline{F}{}^\delta d\cQ^\delta$ on edges~$(v^\circ v^\bullet)$ of quads in~$\Omega^\delta$ by using any of the two values $F^\delta(z)$ or $F^\delta(z')$ on quads~$z,z'$ adjacent to~$(v^\circ v^\bullet)$ in order to define the increment along this edge.\\[2pt]
(ii) One can similarly define a closed differential form $\frac12(\Im[(F^\delta)^2d\cS^\delta]+|F^\delta|^2d\cQ^\delta)$. The primitive~$H^\delta$ of this differential form satisfies the maximum principle.\\[2pt]
(iii) If $F^\delta$ is the basic fermionic observable in a Dobrushin domain~$\Omega^\delta$, then one can choose the global additive constant in the definition of~$H^\delta$ so that $H^\delta=0$ on the free boundary arc~$(ab)_{\Omega^\delta}$ and~$H^\delta=1$ on the wired boundary arc~$(ba)_{ \Omega^\delta}$. In particular, one has $0\le H^\delta\le 1$ everywhere in~$\Omega^\delta$.
\end{proposition}
\begin{proof} (i) Since $d\cQ^\delta((v^\circ v^\bullet))=|d\cS^\delta((v^\circ v^\bullet))|$ by the definition of s-surfaces, it is easy to see from~\eqref{eq:Pr=Pr} that the expression $F^\delta(z)d\cS^\delta((v^\circ v^\bullet))+ i\overline{F}{}^\delta(z)d\cQ^\delta((v^\circ v^\bullet))$ equals the same expression with $F^\delta(z)$ replaced by~$F^\delta(z')$.\\[2pt]
(ii) The fact that this differential form is well defined can be seen from the identity
\[
\tfrac12(\Im[(F^\delta)^2d\cS^\delta]+|F^\delta|^2d\cQ^\delta)=\tfrac{1}{4i}(F^\delta d\cS^\delta+i\overline{F}{}^\delta d\cQ^\delta)^2\cdot (d\cS^\delta)^{-1}
\]
that holds on each edge~$(v^\circ v^\bullet)$ of $\Omega^\delta$. The fact that the primitives $H^\delta$ of such discrete differential forms appearing in the planar Ising model always satisfy the maximum principle is more subtle. This is proved in~\cite[Proposition~2.10]{chelkak-semb}.\\[2pt]
(iii) This is immediate from the definition of the fermionic observables via Kadanoff--Ceva fermions; see~\cite[Corollary~2.12]{chelkak-semb}. If one instead defines~$F^\delta$ \`a la~\cite{smirnov-icm06,chelkak-smirnov-ising,duminil-parafermion} as the expectation of the signed probability that the FK-Ising interface crosses a given edge~$(v^\circ v^\bullet)$ divided by $(-i(\cS^\delta(v^\bullet)-\cS(v^\circ)))^{1/2}$, then the claim boils down to the fact that these probabilities equal~$1$ on the source and the exit edges~$a,b$ of~$\Omega^\delta$.
\end{proof}

\begin{proposition}\label{prop:observables-boundary} 
Let an s-embedding~$\cS^\delta$ have the property~\UnifD\ and $\Omega^\delta$ be a Dobrushin domain on~$\cS^\delta$. Then, the following holds:\\[2pt]
(i) For each boundary half-quad~$z=(v_-^\circ v^\bullet v_+^\circ)$ {on the free arc~$(ab)_{\Omega^\delta}$ one has}
\[
|F^\delta(z)|\,\asymp\,\delta^{-\frac12}\mathbb{P}^{0/1}_{\Omega^\delta}[v^\bullet\leftrightarrow (ba)_{\Omega^\delta}.
\]
{(We write~$f\asymp g$ if~$f=O(g)$ and~$g=O(f)$ with constants that depend on~$c_0$ only.)}\\[2pt]
(ii) Let~$G^\delta$ be the discrete primitive of the form~$F^\delta d\cS^\delta+i\overline{F}{}^\delta d\cQ^\delta$. For each boundary half-quad~$z=(v_-^\circ v^\bullet v_+^\circ)$ on~$(ab)_{\Omega^\delta}$ {one has}
\[
|G^\delta(v_+^\circ)-G^\delta(v_-^\circ)|\,\asymp\,\delta^{\frac12}\mathbb{P}^{0/1}_{\Omega^\delta}[v^\bullet\leftrightarrow (ba)_{\Omega^\delta}].
\]
Moreover, $\arg(G^\delta(v_+^\circ)-G^\delta(v_-^\circ))=-\arg F^\delta(z)$.\\[2pt]
(iii) In addition, assume that the free arc~$(ab)_{\Omega^\delta}=(ab)$ of~$\Omega^\delta$ is a special cut going in the positive direction (see Section~\ref{subsec:cuts} for the definition). Then, for each half-quad $z$ on~$(ab)$ one has $|F^\delta(z)|\asymp \Re F^\delta(z)>0$ and the same is true for~$G^\delta(v_+^\circ)-G^\delta(v_-^\circ)$.
\end{proposition}

\begin{remark} The `positivity' property (iii) is the key motivation for the construction of special cuts from~\cite[Section~5]{chelkak-semb} that may look very artificial otherwise. 
\end{remark}

\begin{proof} (i) Let~$\cX_\pm:=(\cS^\delta(v^\bullet)-\cS^\delta(v_\pm^\circ))^{1/2}$, where the two square roots are chosen so that~$\Im[\cX_+\overline{\cX}_-]>0$. It follows from the definition of the fermionic observable in a Dobrushin domain that the projections~\eqref{eq:Pr=Pr} of the value~$F(z)$ on the directions~\eqref{eq:eta-def} associated with the two edges $(v_\pm^\circ v^\bullet)$ of the half-quad~$z$ equal 
\[
e^{i\frac\pi 4}\cX_\pm^{-1}\cdot \mathbb{P}^{0/1}_{\Omega^\delta}[v^\bullet\leftrightarrow (ba)_{\Omega^\delta}].
\]
A simple computation (see also~\cite[Lemma~5.3]{chelkak-semb}) shows that 
\begin{equation}
\label{eq:F=via-Xpm}
F^\delta(z)=\frac{e^{i\frac{3\pi}4}(\overline{\cX}_+-\overline{\cX}_-)}{\Im[\cX_+\overline{\cX}_-]}\cdot \mathbb{P}^{0/1}_{\Omega^\delta}[v^\bullet\leftrightarrow (ba)_{\Omega^\delta}].
\end{equation}
As~$\cS^\delta$ has property~\UnifD, each of~$\cX_\pm$ is uniformly comparable to~$\delta^{1/2}$ and, moreover, the first factor is uniformly comparable to~$\delta^{-1/2}$.\\[2pt]
(ii) By definition, one has
\begin{align*}
G^\delta(v_+^\circ)-G^\delta(v_-^\circ)
&=F^\delta(z)\cdot(\cX_-^2-\cX_+^2)+i\overline{F}{}^\delta(z)\cdot(|\cX_-|^2-|\cX_+|^2)\\
&=2e^{-i\frac{3\pi}{4}}(\cX_+-\cX_-)\cdot \mathbb{P}^{0/1}_{\Omega^\delta}[v^\bullet\leftrightarrow (ba)_{\Omega^\delta}],
\end{align*}
where the formula~\eqref{eq:F=via-Xpm} and straightforward algebraic manipulations are used to pass from the first to the second lines. \\[2pt]
(iii) Denote $d\cS^\delta(z):=\cS^\delta(v_+^\circ)-\cS^\delta(v_-^\circ)$ and $d\cQ^\delta(z):=\cQ^\delta(v_+^\circ)-\cQ^\delta(v_-^\circ)$. Also, let
\[
u=\tfrac12(\cS^\delta(v_+^\circ)-\cS^\delta(v_-^\circ)),\quad w=\cS^\delta(v^\bullet)-\tfrac12(\cS^\delta(v_-^\circ)+\cS^\delta(v_+^\circ)).
\] 
(In particular, one has $\cX_\pm=(w\mp u)^{1/2}$.) Due to~\eqref{eq:F=via-Xpm}, we need to prove that there exists a (small) constant $\varepsilon>0$ that depends only on the constant~$c_0$ in the property~\UnifD\ such that $|\arg(e^{-i\frac{3\pi}{4}}(\cX_+-\cX_-))|\le\frac\pi 2-\varepsilon$. A straightforward computation via the inverse Zhukovsky function (see \cite[Corollary~5.4]{chelkak-semb}) shows that this property holds in the following three cases:
\begin{enumerate}\renewcommand\theenumi{\alph{enumi}}
\item $\arg u=\arg (d\cS^\delta(z))\in [-\frac\pi 2+\varepsilon\,;\frac\pi2-\varepsilon]$;
\item $\arg (d\cS^\delta(z))\in (\frac\pi2-\varepsilon\,;\frac{3\pi}2-\varepsilon)$ and $|w+u|-|w-u|\le 2|u|\sin(\arg u+\varepsilon)$, which can be written as~$d\cQ^\delta(z)\le \Im[e^{i\varepsilon}d\cS^\delta(z)]$;
\item $\arg (d\cS^\delta(z))\in (-\frac{3\pi}2+\varepsilon\,;-\frac{\pi}2+\varepsilon)$ and $|w+u|-|w-u|\ge 2|u|\sin(\arg u -\varepsilon)$, which can be written as~$d\cQ^\delta(z)\ge \Im[e^{-i\varepsilon}dS^\delta(z)]$.
\end{enumerate}
Recall that the property~\UnifD\ implies that
\[
|dQ^\delta(z)|\le \kappa|dS^\delta(z)|.
\] 
It is now not hard to see that~(b) and~(c), respectively, hold if we have
\begin{enumerate}\renewcommand\theenumi{\alph{enumi}'}\setcounter{enumi}{1}
\item $\arg (d\cS^\delta(z))\in (\frac\pi2-\varepsilon\,;\frac{3\pi}2-m\varepsilon)$ and $\Im[e^{im\varepsilon}(d\cS^\delta(z)-id\cQ^\delta(z))]\ge 0$;
\item $\arg (d\cS^\delta(z))\in (-\frac{3\pi}2+m\varepsilon\,;-\frac{\pi}2+\varepsilon)$ and $\Im[e^{-im\varepsilon}(d\cS^\delta(z)-id\cQ^\delta(z))]\le 0$;
\end{enumerate}
provided that~$\kappa<(1-\frac1m)^2$ and~$\varepsilon$ is small enough. Indeed, to deduce (b) from (b') one notes that
\[
\Im[(e^{i\varepsilon}-e^{im\varepsilon})dS^\delta(z)]\ge(1-\cos(m-1)\varepsilon)|dS^\delta(z)|\ge (1-\cos m\varepsilon)dQ^\delta(z)
\]
provided that~$\kappa\le \sin^2\frac{(m-1)\varepsilon}{2}/\sin^2\frac{m\varepsilon}{2}$. (The derivation of (c) from (c') is similar.) To summarize, the desired property $|\arg(e^{-i\frac{3\pi}{4}}(\cX_+-\cX_-))|\le\frac\pi 2-\varepsilon$ holds if we have
\[
|\arg(dS^\delta(z))|\le \tfrac{3\pi} 2- m\varepsilon\quad \text{and}\quad |\arg(d\cS^\delta(z)-id\cQ^\delta(z))|\le\pi-m\varepsilon.
\]
For fixed~$m$ and small enough $\varepsilon$, this follows from the construction of special cuts; see Section~\ref{subsec:cuts} above.
\end{proof}

\subsection{Regularity of fermionic observables and the proof of Proposition~\ref{prop:upper-bound}} 
\label{subsec:regularity-of-F}
Recall that s-surfaces~$\Theta^\delta$ associated with s-embeddings~$\cS^\delta$ having property~\UnifD\ are automatically $\kappa$-Lipschitz on distances bigger than $C_0\delta$, where constants $C_0>1$ and $\kappa<1$ depend only on the constant~$c_0$ in the {property} \UnifD. In the terminology of~\cite{chelkak-semb}, this means that~$\cS^\delta$ satisfy assumption~\LipK\kappa{C_0\delta}. Moreover, such s-embeddings automatically satisfy an additional technical assumption~\ExpFat{C_0\delta}{\delta'} from the same paper with~$\delta'=C_0^{-1}\delta$ provided that~$C_0$ is chosen big enough. This allows one to apply to observables~$F^\delta$ the a priori regularity theory discussed in~\cite[Section~2.6]{chelkak-semb}. 

\begin{proposition}
\label{prop:regularity_estimates}
There exist constants~$C_1,C_2>0$ and $\beta>0$ that depend only on {the constant $c_0$ in the property~\UnifD}\ such that for each Dobrushin domain~$\Omega^\delta$ on~$\cS^\delta$ the basic fermionic observable in~$\Omega^\delta$ satisfies the estimates
\begin{align}
\label{eq:F-via-HF-Harnack}
|F^\delta(z)|\ &\le\  C_1(\dist(z,\partial\Omega^\delta))^{-\frac12}\\
\osc_{B(z,r)}(F^\delta(z))\ &\le\  C_2r^\beta(\dist(z,\partial\Omega^\delta))^{-\frac12-\beta}
\label{eq:dF-via-HF}
\end{align}   
for each $r\le\frac12\dist(z,\partial\Omega^\delta)$. (In particular, both estimates are uniform in~$\delta$.)
\end{proposition}
\begin{proof} The first estimate follows from~\cite[Corollary~2.20]{chelkak-semb} and the universal bound $0\le H^\delta\le 1$ discussed above in items~(ii),~(iii) of Proposition~\ref{prop:closed-forms}. The second estimate follows from~\cite[Theorem~2.17]{chelkak-semb}.
\end{proof}

We are now in the position to give a proof of Proposition~\ref{prop:upper-bound}, which follows from Proposition~\ref{prop:observables-boundary}, the closedness of the differential form $F^\delta d\cS^\delta+i\overline{F}{}^\delta d\cQ^\delta$, and the a priori regularity estimate~\eqref{eq:F-via-HF-Harnack}.

\begin{proof}[Proof of Proposition~\ref{prop:upper-bound}] 
By rotating the s-embedding we can assume, without loss of generality, that the special cut~$(ab)=(ab)_{\Omega^\delta}$ goes in the positive direction (see Section~\ref{subsec:cuts} for the definition). It is also convenient to rescale the original s-embedding~$\cS$ and Dobrushin domain~$\Omega$ so that~$(ab)$ {has $\|\cdot\|_{2,1}$-length~$1$.} (However, note that the forthcoming proof does include the passage to the limit~$\delta\to 0$ and thus can be run in the original setup~$\delta=1$ as well.) 

Due to items~(ii) and~(iii) of Proposition~\ref{prop:observables-boundary}, we have
\begin{equation}
\label{eq:x-sumP=int}
\sum_{z\in(ab)}\mathbb{P}^{0/1}_{\Omega^\delta}[v^\bullet(z)\leftrightarrow (ba)_{\Omega^\delta}]\ \asymp\ \delta^{-\frac12}\sum_{z\in (ab)}\Re\big[F^\delta(z)d\cS^\delta(z)+i\overline{F}{}^\delta(z)d\cQ^\delta(z)\big].
\end{equation}
Recall that~$(ab)$ is the graph of a $\kappa(1\!-\!\kappa^2)^{-1/2}$-Lipschitz function. Let
\[
z^\sharp:=z+\tfrac{1}{2} i{\eta}\cdot\min\{\|z-a\|_{2,1};\|z-b\|_{2,1}\}\ \ \text{for}\ \ z\in (ab)
\]
and $(ab)^\sharp$ be a discretization of the curve $\{z^\sharp,\ z\in(ab)\}$. 
Due to our assumptions, the path $(ab)^\sharp$ belongs to~$\Omega^\delta$ and, moreover, for each~$z^\sharp\in (ab)^\sharp$ one has
\[
\dist(z^\sharp,\partial\Omega^\delta)\,\ge\,\mathrm{cst}\cdot {\eta}\min\{\|z-a\|_{2,1};\|z-b\|_{2,1}\},
\]
where $\mathrm{cst}>0$ is a constant that depends {on~$c_0$ only}. Using the closedness of the differential form $F^\delta(z)d\cS^\delta(z)+i\overline{F}{}^\delta(z)d\cQ^\delta$ we can replace the contour of integration in~\eqref{eq:x-sumP=int} by~$(ab)^\sharp$ and apply the estimate~\eqref{eq:F-via-HF-Harnack} to control $|F^\delta(z^\sharp)|$ for~$z^\sharp\in (ab)^\sharp$. Thus,
\[
\sum_{z\in(ab)}\mathbb{P}^{0/1}_{\Omega^\delta}[v^\bullet(z)\leftrightarrow (ba)_{\Omega^\delta}]\ \le\ \mathrm{cst}\cdot \delta^{\frac12}\!\!\!\!\sum_{z^\sharp\in (ab)^\sharp}(\dist(z^\sharp;\partial\Omega^\delta))^{-\frac12}\,=O(\delta^{-\frac12}{\eta^{-\frac12}})
\]
as the second sum discretizes a \emph{convergent} integral and hence is bounded from above by~$O(\delta^{-1})\cdot{O(\eta^{-1/2})}$. Thus, $\sum_{z\in(ab)}\mathbb{P}^{0/1}_{\Omega^\delta}[v^\bullet(z)\leftrightarrow (ba)_{\Omega^\delta}]=O(\delta^{-\frac12}{\eta^{-\frac12}})$. It remains to rescale the s-embedding~$\cS^\delta$ back by the factor $L\asymp \delta^{-1}$.
\end{proof}

\subsection{Scheme of the proof of Proposition~\ref{prop:lower-bound}}
\label{subsec:sketch-lower-bound}
This lower bound is the most technical ingredient of our paper. Let us first present the strategy of its proof before diving into technical details in the next two sections.
\begin{enumerate}\renewcommand\theenumi{\roman{enumi}}
\item The proof goes by contradiction: we assume that there exists a sequence of `bridge-shore' domains for which the uniform lower bound~\eqref{eq:prop-lower-bound} fails. By rescaling, rotating, and shifting them, we can instead work with a sequence of `bridge-shore' domains~$\Omega^\delta$ on s-embeddings~$\cS^\delta$, $\delta\to 0$, such that their free boundary arcs~$(ab)_{\Omega^\delta}=(ab)$ are special cuts going in the positive direction that {have~$\|\cdot\|_{2,1}$-length~$1$ and} approximate the straight segment~$[-\frac12;\frac12]$ in the charts $\cS^\delta-i\cQ^\delta$.
\item Due to the monotonicity with respect to boundary conditions, it is sufficient to prove an estimate similar to~\eqref{eq:prop-lower-bound} for \emph{doubly connected} domains~$\Xi^\delta$ shown in {Fig.~\ref{fig:Xi}.} Viewed in the charts $\cS^\delta-i\cQ^\delta$, these domains are discretizations of the unit disc~$\mathbb{D}$ slit along the {images of the special cuts~$(ab)$ in these charts.} The boundary conditions in~$\Xi^\delta$ are free on the upper side of the slit, wired on the lower side of the slit, and free on the outer boundary.
    
\begin{figure}
\begin{minipage}{0.48\textwidth}
\includegraphics[clip, trim=5.2cm 5cm 4.8cm 12.5cm, width=\textwidth]{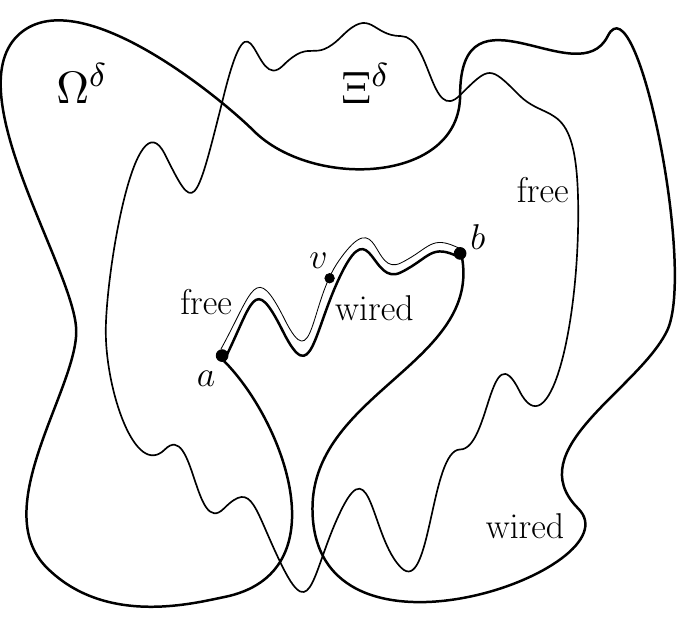}
\end{minipage}
\hskip 0.025\textwidth
\begin{minipage}{0.11\textwidth}
\noindent $\mathop{\xrightarrow{\hspace{36pt}}}\limits^{\displaystyle \cS^\delta-i\cQ^\delta}$
\end{minipage}
\hskip 0.025\textwidth
\begin{minipage}{0.34\textwidth}
\includegraphics[clip, trim=6.85cm 15.8cm 6.6cm 4cm, width=\textwidth]{strongRSW_fig_Xi.pdf}
\end{minipage}
\caption{{\textsc{Left:} A schematic drawing of a Dobrushin domain~$\Omega^\delta$ and a doubly connected discrete domain~$\Xi^\delta$. The free boundary arc~$(ab)_{\Omega^\delta}=(ab)_{\Xi^\delta}=(ab)$ is a special cut. By monotonicity with respect to boundary conditions, the event that a vertex~$v\in (ab)$ is connected to the wired arc~$(ba)_{\Omega^\delta}$ in~$\Omega^\delta$ is more probable than the event that~$v$ is connected to the wired arc~$(ba)_{\Xi^\delta}=(ba)^*$ in~$\Xi^\delta$.
\textsc{Right:} As~$\delta\to 0$, the image $(\cS^\delta-i\cQ^\delta)(\Xi^\delta)$ becomes the unit disc~$\mathbb{D}$ slit along the straight segment~$[-\frac12;\frac12]$, the image of~$(ab)$.
\label{fig:Xi}}}
\end{figure}

\item Recall that the corresponding s-surfaces $\Theta^\delta=(\cS^\delta,\cQ^\delta)$ are $\kappa$-Lipschitz, \mbox{$\kappa<1$,} on scales above~$ C\delta$. By passing to a subsequence one can assume that~$\Theta^\delta$ converge, as~$\delta\to 0$, to a $\kappa$-Lipschitz (in general, non-smooth) surface~$\Theta=\{(z,\vartheta(z))\}\subset\R^{2,1}$ and~$\Xi^\delta$ converge to a doubly connected domain~$\Xi$. Denote by~$\gamma$ the limit of {boundary arcs~$(ab)_{\Omega^\delta}$} as~$\delta\to 0$. If~$w(z)=z-i\vartheta(z)$, then $w(\gamma)=[-\frac12;\frac12]$ and $w(\Xi)=\mathbb{D}\smallsetminus [-\frac12;\frac12]$.
\item Even though~$\Xi^\delta$ is doubly connected (and, contrary to~$\Omega^\delta$, is not a Dobrushin domain), one can define fermionic observables~$F^\delta$ in~$\Xi^\delta$. Moreover, $F^\delta$ have the same `positivity' properties listed in Proposition~\ref{prop:observables-boundary} on the special cut~$(ab)$ as in the case of Dobrushin domains. The corresponding functions~$H^\delta$ are also well defined, satisfy the usual boundary conditions~$H^\delta=0$ on~$(ab)_{\Xi^\delta}=\gamma^+$, by which we mean the upper side of the slit, and~$H^\delta=1$ on~$(ba)_{ \Xi^\delta}=\gamma^-$, by which we mean the lower side of the slit. Moreover, the same universal bound~$0\le H^\delta\le 1$ holds everywhere in~$\Xi^\delta$. 
\item Let~$G^\delta$ be the primitive of the differential form~$F^\delta d\cS^\delta+i\overline{F}{}^\delta d\cQ^\delta$ chosen so that $G^\delta(z_0)=0$, where $z_0$ is a fixed interior point of~$\Xi$. The a priori regularity estimates from Proposition~\ref{prop:regularity_estimates} guarantee the existence of subsequential limits~$F^\delta\to f$, $G^\delta\to g$, $H^\delta\to g$ as~$\delta\to 0$. The convergence is uniform on compact subsets of~$\Xi$ and one has the same relations
    \[
    \textstyle g=\int(fdz+i\overline{f}d\vartheta)\quad \text{and}\quad h=\int\frac12(\Im[f^2dz]+|f|^2d\vartheta)
    \]
    as before. Also, it follows from the uniform estimate~\eqref{eq:F-via-HF-Harnack} that 
\begin{equation}
\label{eq:f=O(d-1/2)}
|f(z)|=O((\dist(z;\partial\Xi))^{-1/2}).
\end{equation} 
In particular, the function~$g$ can be continuously extended to~$\gamma^+$. Moreover, since we assume that the uniform lower bound~\eqref{eq:prop-lower-bound} fails as~$\delta\to 0$, it follows from Proposition~\ref{prop:observables-boundary} that~$g$ must be \emph{constant} on~$\gamma^+$. 
    
\item It is now convenient to view~$g$ as a function of the complex coordinate $w(z)=z-i\vartheta(z)$. It is easy to see that the mapping~$w\mapsto g(w)$ is quasi-analytic inside the domain~$\mathbb{D}\smallsetminus w(\gamma)$. Moreover, $g$ is continuous up to the segment~$w(\gamma^+)=[-\frac12;\frac12]^+$ and is constant there. It also follows from~\eqref{eq:f=O(d-1/2)} that~$g\in W^{1,2-\varepsilon}$ near this segment for each~$\varepsilon>0$. Using standard results from the theory of quasi-analytic mappings it is then not hard to see that~$g$ must be a constant function in~$\Xi$.
\item It follows from the previous step that each subsequential limit $f$ of the functions~$F^\delta$ must vanish, which means that each subsequential limit $h$ of functions~$H^\delta$ has to be constant \emph{inside}~$\Xi$. However, the functions~$H^\delta$ have boundary conditions~$0$ on~$(ab)_{ \Xi^\delta}$ and~$1$ on~$(ba)_{\Xi^\delta}$. The last step of our argument is an analysis, somewhat similar to~\cite[Section~5.5]{chelkak-semb}, of the boundary behaviour of~$H^\delta$ near the upper \emph{and} the lower sides of the slit~$\gamma$, which allows us to prove that, at least on the average,~$H^\delta$ should be close to~$0$ near~$\gamma^+$ and close to $1$ near~$\gamma^-$. (Let us emphasize that, unlike for~$g$, there is no straightforward way to compare the boundary values of~$H^\delta$ and~$h$ as the a priori estimates~\eqref{eq:F-via-HF-Harnack}, \eqref{eq:f=O(d-1/2)} are too weak for this purpose.) This allows us to obtain a contradiction and to complete the proof.
\end{enumerate}

\subsection{Justification of steps (i)--(vi) in the proof of Proposition~\ref{prop:lower-bound}} 
\label{subsec:steps-(i)-(vi)}
We now give details of the proof of Proposition~\ref{prop:lower-bound} that was sketched in Section~\ref{subsec:sketch-lower-bound} except the last step (vii), which is postponed until the next section.

\subsubsection{Steps (i)--(iii): doubly connected domains~$\Xi^\delta$ and~$\Xi$}
\label{subsub:(i)-(iii)} Recall that the proof of Proposition~\ref{prop:lower-bound} goes by contradiction: we assume that there exists a sequence of `bridge-shore' domains~$\Omega^\delta$ on s-embeddings~$\cS^\delta$ such that 
\begin{equation}
\label{eq:sumPinOmega->0}
\textstyle \delta^{\frac12}\sum_{z\in (ab)}\mathbb{P}^{0/1}_{\Omega^\delta}[v^\bullet(z)\leftrightarrow(ba)_{\Omega^\delta}]\,\to\,0\ \ \text{as}\ \ \delta\to 0
\end{equation}
and the {free} boundary arcs~$(ab)_{\Omega^\delta}$ of~$\Omega^\delta$ are special cuts going in the positive direction that approximate the straight segment~$[-\frac12;\frac12]$ in the charts~$\cS^\delta-i\cQ^\delta$. We now define doubly connected domains~$\Xi^\delta$ as follows: \emph{in the same chart}~$\cS^\delta-i\cQ^\delta$, the domain~$\Xi^\delta$ is a discretization of the unit disc~$\mathbb{D}$ slit along the free boundary arc~$(ab)_{ \Omega^\delta}$ such that $\Xi^\delta$ has
\begin{itemize}
\item {free} boundary conditions on the outer boundary (the combinatorial structure of this boundary is irrelevant and can be chosen arbitrarily);
\item {free} boundary condition on the upper side~$(ab)_{\Xi^\delta}=(ab)_{\Omega^\delta}=(ab)$ of the slit, which is kept exactly the same as in~$\Omega^\delta$;
\item {wired} boundary conditions on the lower side~$(ba)_{\Xi^\delta}=(ba)^*$ of the slit that is chosen to be a \emph{dual special cut} {from~$b$ to~$a$} going in the negative direction;
\end{itemize}
see also {Fig.~\ref{fig:Xi}.}

We claim that the monotonicity of the model with respect to the boundary conditions (which follows from the FKG inequality that is available for all ferromagnetic Ising models) and the assumption~\eqref{eq:sumPinOmega->0} imply that
\begin{equation}
\label{eq:sumPinXi->0}
\textstyle \delta^{\frac12}\sum_{z\in (ab)}\mathbb{P}^{0/1}_{\Xi^\delta}\big[v^\bullet(z)\leftrightarrow {(ba)^*}\big]\,\to\,0\ \ \text{as}\ \ \delta\to 0.
\end{equation}
Indeed, when replacing~$\Omega^\delta$ with~$\Xi^\delta$ one decreases each of the terms in this sum: first, imposing the free boundary conditions on the outer boundary of~$\Xi^\delta$ and, second, moving the wired boundary conditions `as far as possible' from~$(ab)$ -- that is, to the other side of~$(ab)$ -- make all connection probabilities smaller. {(Recall also that we assume that~$\Omega^\delta$ does not contain exceptional vertices that lie below the special cut~$(ab)$ but above the dual special cut~$(ba)^*$ and hence~$\Xi^\delta$ contains all vertices near~$(ab)$ that might belong to~$\Omega^\delta$.)}

In order to prove Proposition~\ref{prop:lower-bound} we aim to show that~\eqref{eq:sumPinOmega->0} is impossible, thus obtaining a contradiction. As explained in Section~\ref{subsec:sketch-lower-bound} above, this heavily relies on the analysis of the limits of fermionic observables in domains~$\Xi^\delta$ as~$\delta\to 0$: in fact, the desired contradiction is obtained by considering their (subsequential) limits as~$\delta\to 0$. It is thus convenient to pass, from now onwards, to a subsequence of s-embeddings~$\cS^\delta$ such that the corresponding s-surfaces~$\Theta^\delta$ converge, as~$\delta\to 0$, to a $\kappa$-Lipshitz, $\kappa<1$, space-like surface~$\Theta=(z,\vartheta(z))\subset\R^{2,1}$. In particular, this makes discrete domains~$\Xi^\delta$ converging to a doubly connected domain~$\Xi\subset\C$ such that $w(\Xi)=\mathbb{D}\smallsetminus [-\frac12;\frac12]$, where~$w(z)=z-i\vartheta(z)$.

\subsubsection{Step (iv): fermionic observables in~$\Xi^\delta$} 
\label{subsub:(iv)}
Let us now discuss fermionic observables~$F^\delta$ in~$\Xi^\delta$. Using the language of spin-disorder operators and Kadanoff--Ceva fermions~\cite{kadanoff-ceva,chelkak-semb,CHI-mixed}, the values of these observables on corners~$c\in\Upsilon^\times(\Xi^\delta)$ can be written as~$X^\delta(c)=\langle\psi_c\mu_{(ab)}\sigma_{(ba)^*}\rangle$ exactly as in Dobrushin domains. (From the combinatorial perspective, one treats the upper side~$(ab)=\gamma^+$ of the slit as a single dual vertex, the lower side~${(ba)^*}=\gamma^-$ as a single primal vertex, and the outer boundary of~$\Xi^\delta$ as another single dual vertex.) Due to the presence of \emph{both} spin and disorder operator sitting on~$\gamma$, the branching structure of~$X^\delta(c)$ when~$c$ makes a full turn around~$\gamma$ is the same as it was in absence of the slit. This allows one to define the complex values~$F^\delta(z)$ {as in Dobrushin domains via}~\cite[Eq.~(1.10) and Proposition~2.5]{chelkak-semb}) so that {the function}~$F^\delta$ does \emph{not} change the sign {when}~$z$ goes around~$\gamma$. {(See \cite[Section~2.2]{chelkak-icm18} and~\cite[Sections~2.1--2.3]{CHI-mixed} for a general discussion of the branching structure of Kadanoff-Ceva fermionic observables.)} A similar consideration of the signs applies if one defines~$F^\delta$ \`a la~\cite{smirnov-icm06,chelkak-smirnov-ising,duminil-parafermion} via the expectation of the signed probability that the FK-Ising interface crosses a given edge~$(v^\circ v^\bullet)$.

It is important that~$F^\delta$ has the same `positivity' properties on~$\gamma^+$ that are listed in Proposition~\ref{prop:observables-boundary} as in the case of Dobrushin domains. Indeed, the proof of Proposition~\ref{prop:observables-boundary} does not use any information about the topology of the domain and only requires that the two projections~\eqref{eq:Pr=Pr} of the value~$F^\delta(z)$ on the directions~\eqref{eq:eta-def} associated with the two edges $(v_\pm^\circ v^\bullet)$ equal $e^{i\frac\pi 4}\cX_\pm^{-1}\cdot \mathbb{P}^{0/1}_{\Xi^\delta}[v^\bullet\leftrightarrow {(ba)^*}]$. This is straightforward from the definition of~$F^\delta$ via the Kadanoff--Ceva fermions. (If one instead uses the definition \`a la~\cite{smirnov-icm06,chelkak-smirnov-ising,duminil-parafermion}, the key observation is that the winding of the interface going from~$a$ to~$b$ and touching~$z\in (ab)$ is uniquely determined modulo~$4\pi$.) In particular, due to this version of Proposition~\ref{prop:observables-boundary} for doubly connected domains~$\Xi^\delta$, the assumption~\eqref{eq:sumPinOmega->0} can be rewritten as
\begin{equation}
\label{eq:sumF->0}
\textstyle \delta\cdot \sum_{z\in (ab)}|F^\delta(z)|\to 0\ \ \text{as}\ \ \delta\to 0,
\end{equation}
where $F^\delta$ stands for the fermionic observable in~$\Xi^\delta$. Moreover, an analogue of Proposition~\ref{prop:observables-boundary} holds on the lower side~$\gamma^-$ of the slit since~$\gamma^-$ is a special cut for the dual model as discussed in Section~\ref{subsec:cuts} and Section~\ref{subsub:(i)-(iii)} above. This property of~$\gamma^-$ eventually allows us to claim that the near-to-boundary analysis given in Section~\ref{subsec:step(vii)} can be repeated near~$\gamma^-$ in order to obtain the final contradiction.

Recall now that the definition~\cite[Section~2.4]{chelkak-semb} of the function~$H^\delta$ corresponding to the fermionic observable~$F^\delta$ is purely combinatorial. Similarly to Dobrusin domains, it readily follows from this definition that~$H^\delta|_{(ab)}=0$ and~$H^\delta|_{(ba)^*}=1$ under an appropriate choice of a global additive constant. In particular, the function~$H^\delta$ does \emph{not} have additive monodromy around the slit~$\gamma$. Moreover, one has the usual a priori estimate $0\le H^\delta\le 1$ due to the general `combinatorial' maximum principle provided by~\cite[Proposition~2.10]{chelkak-semb}, which implies that the only possible extremal points of~$H^\delta$ are~$(ab)$ and~$(ba)^*$ but \emph{not} the outer boundary of~$\Xi^\delta$.

\subsubsection{Steps (v)-(vi): boundary values of the function~$g$ on~$\gamma^+$ and uniqueness} \label{subsub:(v)-(vi)}
Denote by~$G^\delta$ the primitive of the closed differential form~$F^\delta d\cS^\delta+i\overline{F}{}^\delta(z)d\cQ^\delta$ defined in an (upper) vicinity of the {special cut}~$(ab)$. As in the case of Dobrushin domains, the estimate~$0\le H^\delta\le 1$ in~$\Xi^\delta$ implies, via~\cite[Corollary~2.20]{chelkak-semb}, that 
\begin{equation}
\label{eq:Fdelta=O(d-1/2)}
F^\delta(z)=O(\dist(z,\partial\Omega^\delta)^{-\frac12}),
\end{equation}
uniformly in~$\delta$. Therefore, the functions~$G^\delta$ are uniformly $\frac12$-H\"older up to~$(ab)=\gamma^+$. Indeed, to compare the values of~$G^\delta$ at two points~$z_1,z_2$ with $|z_2-z_1|=d$, one can choose a contour of integration that first goes from~$z_1$ vertically upwards for a distance comparable to~$d$, then propagates inside a $d$-interior of~$\Xi^\delta$, and reaches~$z_2$ by going vertically downwards. 

It follows from the assumption~\eqref{eq:sumF->0} that
\[
\osc_{(ab)}(G^\delta)\to 0\quad \text{as}\ \ \delta\to 0.
\] 
Due to the uniform continuity of~$G^\delta$ up to~$(ab)$ discussed above and the uniform convergence~$G^\delta\to g$ on compact subsets of~$\Xi$, this implies that the function~$g$ is constant on~$(ab)$.

Let us now consider $g$ as a function of the variable~$w(z)=z-i\vartheta(z)$ defined in an upper vicinity of the slit~$w(\gamma)=[-\frac12;\frac12]$. Recall that $\vartheta$ is a $\kappa$-Lipschitz, $\kappa<1$, real-valued function of~$z$ and hence 
\[
|\vartheta_z|=|\vartheta_{\bar{z}}|\le\tfrac12\kappa.
\]
In particular, $z\mapsto w$ is a bi-Lipshitz homeomorphism and $|z_{\bar{w}}|\le \kappa|z_w|$ almost everywhere. (In fact, one has $|z_{\bar{w}}|\le\kappa'|z_w|$ with $\kappa'=\kappa/(2-\kappa)<\kappa$.) Since $g=\int(fdz+i\overline{f}d\vartheta)$, this function is almost everywhere differentiable in~$z$ and one also has $|g_{\bar{z}}|\le \kappa |g_z|$. Therefore,
\begin{equation}
\label{eq:g-quasireg}
|g_{\bar{w}}|\le \frac{2\kappa}{1+\kappa^2}|g_w|
\end{equation}
almost everywhere, which is equivalent to saying that~$g$ is a weakly $K^2$-quasiregular function of $w$ where~$K=(1+\kappa)/(1-\kappa)$ as usual.

Passing to the limit $\delta\to 0$ in~\eqref{eq:Fdelta=O(d-1/2)} one sees that $f(z)=O(\dist(z,\partial\Xi)^{-1/2})$ and hence the function $z\mapsto g(z)$ belongs to the Sobolev space $W^{1,2-\varepsilon}$ in an upper vicinity of the slit~$\gamma$ for each $\varepsilon>0$. Since the mapping $z\mapsto w(z)$ is bi-Lipschitz, the same is true if one views $g$ as a function of $w$. Recall that~$g=c$ on the segment $w(\gamma)=[-\frac12;\frac12]$ for some constant~$c\in\C$. Using the Schwarz reflection $g(\overline{w}):=-\overline{g(w)}+2\Re c$, one can extend $g$ to a weakly $K^2$-regular function defined in a vicinity of the origin such that~$g\in W^{1,2-\varepsilon}_\mathrm{loc}$ for $\varepsilon<\frac{2}{K^2+1}$. Due to a classical result~\cite[Corollary~14.4.3]{astala-book} from the theory of quasiregular functions, these two conditions imply that $g\in W^{1,2}_\mathrm{loc}$ as a function of~$w$. This further implies that $g$ admits a Sto\"{\i}low factorization (e.g., see \cite[Corollary~14.4.5]{astala-book}) and hence cannot be constant on a segment unless it is constant everywhere in~$\Xi$, which means that~$f=0$ everywhere in~$\Xi$.

\subsection{Justification of step (vii) in the proof of Proposition~\ref{prop:lower-bound}} 
\label{subsec:step(vii)}
It remains to rule out the scenario~$f=0$ and~$h=\mathrm{cst}$ everywhere in~$\Xi$. Our strategy is similar to the analysis performed in~\cite[Section~5]{chelkak-semb} with several simplifications {that allow} us to go beyond the particular case of `almost horizontal' s-surfaces considered in~\cite{chelkak-semb}.

In what follows we always work on S-graphs $\cS^\delta-i\cQ^\delta$; recall that in this chart the special { cuts~$(ab)_{\Xi^\delta}=(ab)_{\Omega^\delta}=(ab)$ approximate} the segment~$[-\frac12;\frac12]\subset\R$. As discussed in \cite[Section~2.5]{chelkak-semb} (in particular, see~\cite[Proposition~2.16]{chelkak-semb}), the function~$\Re F^\delta(z)$ is a martingale with respect to a natural continuous-time \emph{backward} random walk~$\widetilde{Z}_t$ on $\cS^\delta-i\cQ^\delta$. Moreover, as discussed in~\cite[Definition~5.5., Lemma~5.6 and Lemma~5.7]{chelkak-semb}, this random walk can be modified near~$(ab)$ so that it stops on $(ab)$. Similarly to~\cite{chelkak-semb}, we use these observations to prove the following integral estimate for the functions~$\Re F^\delta(z)$. For~$0<y\ll 1$, denote
\[
\cR(y):=\{z\in\Xi^\delta\ \text{that lie above~$(ab)$ with}\ |\Re \cS^\delta(z)|\le \tfrac14,\ \Im (\cS^\delta(z)-i\cQ^\delta)(z)\le y\}\,.
\]
\begin{lemma} 
\label{lem:ReF-in-R(y)}
For small enough~$0<y\ll 1$, the estimate
\begin{equation}
\label{eq:ReF-in-R(y)}
\delta^2\!\!\!\!\sum_{z\in \cR(2y)\smallsetminus\cR(y)}|\Re F^\delta(z)|\,=\,O(y)
\end{equation}
holds uniformly in~$\delta$.
\end{lemma}
\begin{proof} Fix a sufficiently small $h_0>0$ and let  $\cT$ be the isosceles triangle with bottom side $[-\frac12;\frac12]$ and the opposite vertex~$ih_0$ drawn in the chart~$\cS^\delta-i\cQ^\delta$. 
Denote by~$T$ the exit time from~$\cT$ for the backward random walk~$\widetilde{Z}^z_t$ started at a point~$z$. We will be applying the optional stopping theorem for the martingale~$\Re F^\delta(\widetilde{Z}^z_t)$. If~$h_0$ is chosen small enough, then so are the angles of the triangle~$\cT$. Then, the uniform crossing estimates for the backward random walk~$\widetilde{Z}_t$ give {an upper bound}
\[
\mathbb{P}\big[\,{\widetilde{Z}^z_T\not\in(ab)\ \&\ }\Im (\cS^\delta-i\cQ^\delta)(\widetilde{Z}^z_T)\le s\,\big]=O({s})
\]
for all {$z$ with $|\Re z|\le\frac14$ and $s>0$. (Moreover, choosing~$h_0$ small enough, one can obtain an upper bound~$O(s^p)$ with arbitrarily large~$p$ in the right-hand side. Any~$p>\frac12$ is sufficient for our purposes.)} Together with~\eqref{eq:Fdelta=O(d-1/2)}, this estimate implies that
\[
\Re F^\delta(z)=O(1)+\mathbb{E}\big[\Re F^\delta(\widetilde{Z}_T^z)\mathbbm{1}_{\widetilde{Z}_T^z\in(ab)}\big]\,.
\]

Due to Proposition~\ref{prop:observables-boundary}, one has~$\Re F^\delta>0$ everywhere on~$(ab)$. {Note that the rectangle~$\cR(2y)\smallsetminus\cR(y)$ contains~$O(\delta^{-2}y)$ vertices.} Therefore,
\begin{align*}
\delta^2\!\!\!\!\sum_{z\in \cR(2y)\smallsetminus\cR(y)}|\Re F^\delta(z)|&\,=\, O(y)+\delta^2\!\!\sum_{z\in\cR(2y)\smallsetminus\cR(y)}\sum_{z_0\in(ab)}\mathbb{P}\big[\widetilde{Z}_T^z=z_0\big]\Re F^\delta(z_0)\\
&\,=\,O(y)+\delta^2\!\!\!\!\sum_{z_0\in(ab)}\Re F^\delta(z_0)\sum_{z\in\cR(2y)\smallsetminus\cR(y)}\mathbb{P}\big[\widetilde{Z}_T^z=z_0\big].
\end{align*}
It follows from Proposition~\ref{prop:upper-bound} and Proposition~\ref{prop:observables-boundary} that
\[
\textstyle \delta\sum_{z_0\in(ab)}\Re F^\delta(z_0)=O(1).
\]
(In fact, when arguing by contradiction we even assume that this sum tends to zero as~$\delta\to 0$, see~\eqref{eq:sumF->0}. However, let us emphasize that we do \emph{not} use this stronger property in what follows. This is important as we also need the same near-to-boundary analysis to be applicable near the lower side of the cut~$\gamma$, for which an analogue of~\eqref{eq:sumF->0} is not available.) Thus, in order to prove the lemma it is sufficient to show that
\begin{equation}
\label{eq:x-sumP=O(y)}
\textstyle \delta\sum_{z\in\cR(2y)\smallsetminus\cR(y)}\mathbb{P}\big[\widetilde{Z}_T^z=z_0\big]=O(y)
\end{equation}
uniformly with respect to~$\delta$ and $z_0\in(ab)$.

To prove~\eqref{eq:x-sumP=O(y)}, let us now consider the \emph{forward} random walk~$Z_t^{z_0}$ on~$\cS^\delta-i\cQ^\delta$ started at the boundary point~$z_0\in(ab)$ and stopped when exiting the triangle~$\cT$. {For a fixed~$z_0\in (ab)$, the sum $\sum_{z\in\cR(2y)\smallsetminus\cR(y)}\mathbb{P}\big[\widetilde{Z}_T^z=z_0\big]$ equals, up to a factor comparable to~$1$, the expected number of vertices on the trajectory of~$Z_t^{z_0}$ that lie inside the rectangle~$\cR(2y)\smallsetminus\cR(y)$.} Recall (see~\cite[Section~2.5]{chelkak-semb} for details) that the position of this continuous-time random walk on~$\cS^\delta-i\cQ^\delta$, viewed as a complex number, is a martingale and its time-normalization is chosen so that~$|(S^\delta-i\cQ^\delta)(Z_t)|^2-t$ is a martingale. In particular, it spends $O(\delta^2)$ time, on average, at each vertex along its path. Therefore, the desired estimate~\eqref{eq:x-sumP=O(y)} follows from the uniform estimate
\begin{equation}
\label{eq:x-E-time-in-R(y)}
\textstyle \mathbb{E}\big[\int_0^T\mathbbm{1}_{Z_t^{z_0}\in \cR(2y)\smallsetminus\cR(y)}dt\big]=O(\delta y)
\end{equation}
of the expected time spent by~$Z_t^{z_0}$ in~$\cR(2y)\smallsetminus\cR(y)$. (As above, we denote by $T$ the exit time from the triangle~$\cT$.)

To prove~\eqref{eq:x-E-time-in-R(y)}, note that the {property} \UnifD~implies that for small enough $c>0$ the process
\[
(1-c)(\Im (S^\delta-i\cQ^\delta)(Z_t))^2-c(\Re(S^\delta-i\cQ^\delta)(Z_t))^2
\] 
is a submartingale. ({This is easy to check for each jump of the continuous-time random walk~$Z_t$.} Even if one side of the image of a tangential quad of size~$\delta$ on the S-graph~$S^\delta-i\cQ^\delta$ is almost horizontal and then necessarily has length~$\varepsilon\ll\delta$, the random walk~$(\cS^\delta-i\cQ^\delta)(Z_t)$, being a martingale, {makes} the other two jumps with probabilities of order~$\varepsilon\delta^{-1}$. Therefore, the main contribution to the quadratic variation {is of order~$\varepsilon\delta$ and} comes from these two `long' jumps rather than from the {`short' one, which contributes only~$O(\varepsilon^2)$.}) As~$|(S^\delta-i\cQ^\delta)(Z_t)|^2-t$ is a martingale, this implies that the process
\[
(\Im (S^\delta-i\cQ^\delta)(Z_t))^2-ct
\]
is a submartingale too. It is now easy to see that the expected time spent by the martingale process~$\Im(S^\delta-i\cQ^\delta)(Z_t^{z_0})$ in the segment~$[y;2y]$ is bounded by~$O(\delta y)$. Indeed, let~$C\delta$ be the maximal size of its jumps and $\psi:[-C\delta;h_0+C\delta]\to\R_-$ be the function that vanishes at~$-C\delta$ and~$h_0+C\delta$, is linear on the segments $[-C\delta;y-C\delta]$ and $[2y+C\delta;h_0+C\delta]$, and has constant second derivative~$2$ on the segment $[y-C\delta;2y+C\delta]$. {Note that} the process
\[
\textstyle \psi(\Im(S^\delta-i\cQ^\delta)(Z_t^{z_0}))-c\int_0^t\mathbbm{1}_{Z_s^{z_0}\in {\cR(2y)\smallsetminus \cR(y)}}ds
\]
is a submartingale. {Therefore, it follows from the optional stopping theorem that
\[
\psi\big(\Im(\cS^\delta(z_0)-i\cQ^\delta)(z_0)\big)\,\le\,\textstyle - c\cdot\mathbb{E}\big[\int_0^T\mathbbm{1}_{Z_t^{z_0}\in {\cR(2y)\smallsetminus \cR(y)}}dt\big].
\]
Moreover, it is clear that
\[
\psi'|_{[-C\delta;y-C\delta]}\,=\,-2(y+2C\delta)+\psi'|_{[2y+C\delta;h_0+C\delta]}\,\ge\,-2(y+2C\delta)
\]
and hence~$\psi(\Im(\cS^\delta-i\cQ^\delta)(z_0))=O(\delta y)$. This implies the upper bound~\eqref{eq:x-E-time-in-R(y)}.} 
\end{proof}
\begin{corollary}
\label{cor:F|ReF|-in-R(y)}
For all small enough~$0<y\ll 1$ one has the uniform (in~$\delta$ and~$y$) estimate
\begin{equation}
\label{eq:F|ReF|-in-R(y)}
\delta^2\sum_{z\in\cR(2y)}|\Re F^\delta(z)||F^\delta(z)|\,=\,O(y^{\frac12}).
\end{equation}
\end{corollary}
\begin{proof}
The estimates~\eqref{eq:ReF-in-R(y)} and~\eqref{eq:Fdelta=O(d-1/2)} imply that the sum of these terms over the rectangle $\cR(2^{-n+1}y)\smallsetminus\cR(2^{-n}y)$ is $O((2^{-n}y)^{1/2})$, uniformly in~$\delta$ and~$n$. Summing this over~$n=0,1,\ldots,O(|\log \delta|)$ one obtains~\eqref{eq:F|ReF|-in-R(y)}.
\end{proof}

We now use~\eqref{eq:F|ReF|-in-R(y)} to obtain an estimate for the average values of the functions~$H^\delta$ in rectangles~$\cR(2y)\smallsetminus \cR(y)$. Unlike in~\cite[Section~5]{chelkak-semb}, we do \emph{not} use the tilted s-embeddings~$\cS^\delta-ie^{2i\phi}\cQ^\delta$ in order to obtain this estimate. In particular, this allows us not to include the very technical notion of `doubletons' used in~\cite{chelkak-semb} in the definition of special cuts; { cf. Definition~\ref{def:s-cut} above and \cite[Definition~5.1]{chelkak-semb}.}

\begin{lemma}
\label{lem:H-in-R(y)}
For small enough~$0<y\ll 1$, the estimate
\begin{equation}
\label{eq:H-in-R(y)}
\delta^2\!\!\!\!\sum_{v\in\cR(2y)\smallsetminus\cR(y)} H^\delta(v)\,=\,O(y^{\frac32})
\end{equation}
holds uniformly in~$\delta$.
\end{lemma}
\begin{proof} Recall that the function~$H^\delta$ has zero Dirichlet boundary values on~$(ab)$. Given a vertex~$v\in\cR(2y)\smallsetminus \cR(y)$, let~$\pi^\delta(v)$ be a lattice path~$v_0,v_1,\ldots,v_{n+1}$ going from a certain vertex~$v_0\in(ab)$ to~$v_{n+1}=v$ in an $O(\delta)$ vicinity of the corresponding vertical line such that $v_k$ and~$v_{k+1}$ belong to the same quad~$z_k$ and 
\begin{equation}
\label{eq:x-path-upwards}
\Im(S^\delta-i\cQ^\delta)(v_{k+1})>\Im(S^\delta-i\cQ^\delta)(v_k)
\end{equation} 
for all~$k=0,\ldots,n$. (Note that we do \emph{not} insist upon the interlacement of the primal and the dual vertices along~$\pi^\delta(v)$ and allow that~$v_k$ and~$v_{k+1}$ are both primal or both dual ones.) By definition, one has
\[
\textstyle H^\delta(v)=\sum_{k=0}^{n}\frac12\big(\Im[(F^\delta(z_k))^2(\cS^\delta(v_{k+1})-\cS^\delta(v_k))]+|F^\delta(z_k)|^2(\cQ^\delta(v_{k+1})-\cQ^\delta(v_k))\big)
\]
A simple manipulation shows that
\begin{align*}
\Im[(F^\delta(z_k))^2&(\cS^\delta(v_{k+1})-\cS^\delta(v_k))]+|F^\delta(z_k)|^2(\cQ^\delta(v_{k+1})-\cQ^\delta(v_k))\\
=\ & \Re[(F^\delta(z_k))^2]\cdot \Im[(S^\delta-i\cQ^\delta)(v_{k+1})-\Im(S^\delta-i\cQ^\delta)(v_k)]\\
&+\ 2\Re F^\delta(z_k)\Im F^\delta(z_k)\cdot \Re[\cS^\delta(v_{k+1})-\cS^\delta(v_k)]\\
&+\ 2(\Re F^\delta(z_k))^2\cdot (\cQ^\delta(v_{k+1})-\cQ^\delta(v_k))\,.
\end{align*}
Since we choose the path~$\pi^\delta(v)$ so that~\eqref{eq:x-path-upwards} holds, in the first term one can replace $\Re[(F^\delta(z_k))^2]$ by a bigger quantity~$(\Re F^\delta(z_k))^2$, thus getting rid of the problematic term~$(\Im F^\delta(z_k))^2$. This leads to the upper bound
\[
\textstyle H^\delta(v)\ \le\ O(\delta)\cdot \sum_{k=0}^n |\Re F^\delta(z_k)||F^\delta(z_k)|.
\]

Finally, note that for each quad~$z$ we have $O(y\delta^{-1})$ paths~$\pi^\delta(v)$, $v\in\cR(2y)\smallsetminus\cR(y)$ passing through~$z$. (All such possible paths correspond to vertices~$v$ that belong to an~$O(\delta)$ vicinity of the vertical line passing through~$z$.) Hence,
\[
\delta^2\!\!\!\!\sum_{v\in\cR(2y)\smallsetminus\cR(y)} H^\delta(v)\ \le\ O(y)\cdot \delta^2\!\!\!\!\sum_{z\in\cR(2y)\smallsetminus\cR(y)}|\Re F^\delta(z_k)||F^\delta(z_k)|
\]
and the desired estimate~\eqref{eq:H-in-R(y)} readily follows from Corollary~\ref{cor:F|ReF|-in-R(y)}.
\end{proof}

We are now ready to complete the justification of step~(vii) in the proof of Proposition~\ref{prop:lower-bound}. Recall that this proof goes by contradiction and we already know that in this case each subsequential limit~$h$ of functions~$H^\delta$ must be constant everywhere in~$\Xi$. It follows from Lemma~\ref{lem:H-in-R(y)} that the only possible value of this constant is~$0$ as~$H^\delta$ converge to~$h$ uniformly in~$\cR(2y)\smallsetminus\cR(y)$ for each $0<y\ll 1$. However, a similar analysis can be applied near the \emph{lower} side~$\gamma^-$ of the slit. Indeed, as discussed in Section~\ref{subsec:cuts} and Section~\ref{subsub:(i)-(iii)} above, the lower side of~$\gamma$ is a special cut for the \emph{dual} model, which corresponds to the analysis of~$H^\delta$ near the boundary arc with wired boundary conditions. Thus, the same analysis near~$\gamma^-$ implies that the only possible value of the constant function~$h$ is~$1$. This gives the desired contradiction and completes the proof of Proposition~\ref{prop:lower-bound}.

\subsection{Proof of Proposition~\ref{prop:lower-bound-w/o-legs}} \label{subsec:lower-bound-w/o-legs}
Finally, let us show how to derive Proposition~\ref{prop:lower-bound-w/o-legs} from {Proposition~\ref{prop:upper-bound} and Proposition~\ref{prop:lower-bound}} using the closed differential form~$F d\cS+i\overline{F}d\cQ$, the `positivity' of fermionic observables on special cuts provided by Proposition~\ref{prop:observables-boundary}, and the self-duality of the model. This argument is similar to~\cite[Lemma~5.3]{duminil-manolescu-tassion-fractal}.

Due to the finite energy property, we can assume that $L$ is large enough without loss of generality. 
Rotating the s-embedding, we can also assume that~$(ba)_\Omega=(ba)^*$ is a \emph{dual} special cut going in the positive direction; see Sections~\ref{subsec:cuts} and~\ref{subsec:bridge-shore}. Using item~(i) of Proposition~\ref{prop:closed-forms} one gets an approximate identity
\[
\sum_{z\in {(ba)^*}}\Re\big[F(z)d\cS(z)+i\overline{F(z)}d\cQ(z)\big]+\sum_{z\in (ab)_{\Omega}}\Re\big[F(z) d\cS(z)+i\overline{F(z)} d\cQ(z)\big] =O(1),
\]
where the right-hand side is the contribution of the source and the exit edges~$a,b$.
It follows from Proposition~\ref{prop:observables-boundary} and Proposition~\ref{prop:lower-bound} applied to the \emph{dual} model that
\[
\sum_{z\in (ba)^*}\Re\big[F(z)d\cS(z)+i\overline{F(z)}d\cQ(z)\big]\ \asymp\ \!\!\!\sum_{z\in (ba)^*}\mathbb{P}^{0/1}_{\Omega}[v^\circ(z)\,\mathop{\leftrightarrow}\limits^{*}\,(ab)_{\Omega}]
\ \ge\ \mathrm{cst}\cdot L^{\frac12},
\]
where $\mathrm{cst}>0$ is a constant that depends only on the constant~$c_0$ in the property \UnifD\ and $\mathop{\leftrightarrow}\limits^*$ stands for the dual crossing. 

At the same time, due to item~(i) of Proposition~\ref{prop:observables-boundary} we have
\[
\Re\big[F(z)d\cS(z)+i\overline{F(z)}d\cQ(z)\big]\,=\,O\big(\mathbb{P}^{0/1}_{\Omega}\big[v^\bullet(z)\leftrightarrow {(ba)^*}\big]\big)\ \text{for all}\ z\in(ab)_{\Omega}.
\]
Moreover, it follows from Proposition~\ref{prop:lower-bound} {and Lemma~\ref{lem:ramps-not-screened}} applied {for} two `ramp' parts~$(aa_0)$ and~$(b_0b)$ of the free arc~$(ab)_{\Omega}$ that
\[
\textstyle \sum_{z\in (aa_0)\cup(b_0b)} O\big(\mathbb{P}^{0/1}_{\Omega}\big[v^\bullet(z)\leftrightarrow {(ba)^*}\big]\big)\ =\ O\big((\rho L)^{\frac12}\big).
\]
{(Indeed, due to trivial monotonicity reasons, in order to prove this estimate, say, for the special cut~$(aa_0)$ one can consider~$\Omega$ with wired boundary conditions everywhere except~$(aa_0)$ and to apply Proposition~\ref{prop:upper-bound} in this setup.)} Therefore, 
\[
\textstyle O\big(\sum_{z\in (a_0b_0)_{\Omega}}\mathbb{P}^{0/1}_{\Omega}\big[v^\bullet(z)\mathop{\leftrightarrow}{(ba)^*}\big]\big)\ \ge\ \mathrm{cst}\cdot L^{\frac12}-O\big((\rho L)^{\frac12}+1\big),
\]
which proves the desired lower bound provided that~$\rho$ is chosen small enough (depending only on {the constant $c_0$} in the property \Unif) and $L$ is sufficiently large. \qed

\renewcommand{\thesection}{A}
\section{Appendix}
\setcounter{equation}{0}
\renewcommand{\theequation}{A.\arabic{equation}}

We now prove several geometric properties of special cuts that we needed for the construction of `bridge-shore' subdomains in Section~\ref{subsec:fjords}. We begin with formulating a `distorted triangle inequality' for the $\|\cdot\|_{2,1}$-distances between points on $\kappa$-Lipschitz s-surfaces $\cQ^\bullet$ associated with discrete domains~$\cD$.
\begin{lemma}
\label{lem:distorted-triangle-ineq} For each triple of points~$z_0,z_1,z_2\in\cD$ one has
\begin{equation}
\label{eq:distorted-triangle-ineq}
\|z_2-z_0\|_{2,1}\,\le\,\|z_1-z_0\|_{2,1}+\frac{1+\kappa^2}{1-\kappa^2}\|z_2-z_1\|_{2,1}\,.
\end{equation}
\end{lemma}
\begin{proof} 
Without loss of generality, assume that~$z_0=0$ and~$\cQ^\bullet(z_0)=0$. Taking the square and writing the $\|\cdot\|_{2,1}^2$ norms in coordinates, one observes that this inequality follows from the estimate
\begin{align*}
\Re[z_2\overline{z}_1]-\cQ^\bullet(z_2)\cQ^\bullet(z_1)\,&\le\,(1+\kappa^2)|z_1||z_2|\\
&\le\,\frac{1+\kappa^2}{1-\kappa^2}\big(|z_1|^2-(\cQ^\bullet(z_1))^2\big)^{\frac12}\big(|z_2|^2-(\cQ^\bullet(z_2))^2\big)^{\frac12}\,,
\end{align*}
where we used that~$\big|\cQ^\bullet(z_1)\big|\le\kappa|z_1|$ and~$\big|\cQ^\bullet(z_2)\big|\le\kappa|z_2|$.
\end{proof}
We are now ready to prove the key estimate used in Section~\ref{subsec:fjords}.
\begin{lemma} 
\label{lem:polar-coordinates}
Let~$\gamma_1$ and~$\gamma_2$ be two special cuts emanating from the same point~$z_0$. Denote by~$\phi\in(0,\pi]$ the angle between their directions and let~$z_1\in\gamma_1$, $z_2\in\gamma_2$ be two points on these cuts such that~$\|z_1-z_0\|_{2,1}=\|z_2-z_0\|_{2,1}=\ell>0$. Then,
\begin{equation}
\label{eq:polar-coordinates}
\tfrac12\phi\ell\le \dist(z_1,z_2)\le A\phi\ell,
\end{equation}
where the constant $A=(1-\kappa)^{-1}(1-\kappa^2)^{-1/2}$ depends only on~$c_0$.
\end{lemma}
\begin{proof} Without loss of generality, assume that $z_0=0$, $\ell=1$, the special cut~$\gamma_1$ has direction~$e^{-\frac12i\phi}$, and~$\gamma_2$ has direction~$e^{\frac12i\phi}$. By definition of special cuts we have
\[
\begin{array}{lcl}
z_1=e^{-\frac12i\phi}(1+iy_1),&\ & \cQ^\bullet(z_1)=y_1,\\
z_2=e^{\frac12i\phi}(1+iy_2),&& \cQ^\bullet(z_2)=y_2
\end{array}
\]
for some~$y_1,y_2\in\R$ such that $|y_1|,|y_2|\le \kappa (1-\kappa^2)^{-1/2}$. It is easy to see that
\[
\dist(z_1,z_2)\le |y_2-y_1|+2\sin\tfrac\phi2\cdot (1-\kappa^2)^{-1/2}
\]
and hence
\[
\dist(z_1,z_2)-2\sin\tfrac\phi2\cdot(1-\kappa^2)^{-1/2}\,\le\, |y_2-y_1|\,\le\, \kappa\dist(z_1,z_2),
\]
which gives the desired upper bound for~$\dist(z_1,z_2)$ as $2\sin\tfrac\phi2\le\phi$.

In order to prove the lower bound, note that
\[
\dist(z_1,z_2)\,\ge\, |(y_2-\tan\tfrac\pi 2)-(y_1+\tan\tfrac\pi 2)|\cdot \cos\tfrac\phi2,
\]
which implies that 
\[
|\cQ^\bullet(z_2)-\cQ^\bullet(z_1)|\,\ge\,2\tan\tfrac\phi2-|y_2-y_1-2\tan\tfrac\phi 2|\,\ge\,2\tan\tfrac\phi2-\frac{\dist(z_1,z_2)}{\cos\frac\phi2}.
\]
Therefore,
\[
2\sin\tfrac\phi2-\dist(z_1,z_2)\le \kappa\cos\tfrac\phi2\cdot \dist(z_1,z_2),
\]
which implies the claim since~$\kappa<1$ and $2\sin\tfrac\phi2/(1+\cos\tfrac\phi2)=2\tan\tfrac\phi4\ge\tfrac\phi2$.
\end{proof}
\begin{corollary} 
\label{cor:cuts-do-not-screen-each-other} There exists a constant $c=c(c_0)>0$ such that the following holds: in the setup of Lemma~\ref{lem:polar-coordinates} one has $\dist(z_1,\gamma_2)\ge c\phi\ell$ and $\dist(z_2,\gamma_1)\ge c\phi\ell$.
\end{corollary}
\begin{proof} This follows from the lower bound in Lemma~\ref{lem:polar-coordinates} and the `distorted triangle inequality'~\eqref{eq:distorted-triangle-ineq}. Indeed, for each~$z'_2\in B_{c\phi\ell}(z_1)\cap \gamma_2$ we have
\[
\|z'_2-z_2\|_{2,1}\,=\,\big|\|z'_2-z_0\|_{2,1}-\ell\big|\,\le \frac{1+\kappa^2}{1-\kappa^2}\cdot c\phi\ell 
\]
and hence
\[
c\phi\ell\,\ge\,\|z'_2-z_1\|_{2,1}\,\ge\,\tfrac12\phi\ell - \frac{1+\kappa^2}{1-\kappa^2}\|z'_2-z_2\|_{2,1}\,,
\]
which is impossible if $c<\frac14(1-\kappa^2)^2/(1+\kappa^4)$. The same proof applies to~$z_2$.
\end{proof}

\bibliographystyle{plain}
\bibliography{strongRSW}

@book {astala-book,
    AUTHOR = {Astala, Kari and Iwaniec, Tadeusz and Martin, Gaven},
     TITLE = {Elliptic partial differential equations and quasiconformal
              mappings in the plane},
    SERIES = {Princeton Mathematical Series},
    VOLUME = {48},
 PUBLISHER = {Princeton University Press, Princeton, NJ},
      YEAR = {2009},
     PAGES = {xviii+677},
      ISBN = {978-0-691-13777-3},
   MRCLASS = {30C62 (30G20 35J46 35J60 35J92)},
  MRNUMBER = {2472875},
MRREVIEWER = {Olli\ Martio},
}

@article {benoist-hongler,
    AUTHOR = {Benoist, St\'{e}phane and Hongler, Cl\'{e}ment},
     TITLE = {The scaling limit of critical {I}sing interfaces is
              {$\mathrm{CLE}_3$}},
   JOURNAL = {Ann. Probab.},
  FJOURNAL = {The Annals of Probability},
    VOLUME = {47},
      YEAR = {2019},
    NUMBER = {4},
     PAGES = {2049--2086},
      ISSN = {0091-1798,2168-894X},
   MRCLASS = {60J67 (60K35 82B20 82B27)},
  MRNUMBER = {3980915},
MRREVIEWER = {Jianping\ Jiang},
       DOI = {10.1214/18-AOP1301},
       URL = {https://doi.org/10.1214/18-AOP1301},
}

@ARTICLE{cahen-garban-sepuvelda,
       author = {{Cahen}, Paul and {Garban}, Christophe and {Sep{\'u}lveda}, Avelio},
        title = "{The critical Ising magnetization field can be reconstructed from its $+/-$ interfaces}",
      journal = {arXiv e-prints},
         year = 2026,
        month = oct,
          eid = {arXiv:2610.02064},
        pages = {arXiv:2610.02064},
          doi = {10.48550/arXiv.2610.02064},
archivePrefix = {arXiv},
       eprint = {2610.02064},
 primaryClass = {math.PR},
       adsurl = {https://ui.adsabs.harvard.edu/abs/2026arXiv261002064C}
}

@article {camia-feng-Ising,
    AUTHOR = {Camia, Federico and Feng, Yu},
     TITLE = {Conformal covariance of connection probabilities in the 2{D}
              critical {FK}-{I}sing model},
   JOURNAL = {Stochastic Process. Appl.},
  FJOURNAL = {Stochastic Processes and their Applications},
    VOLUME = {189},
      YEAR = {2025},
     PAGES = {Paper No. 104734, 25},
      ISSN = {0304-4149,1879-209X},
   MRCLASS = {82B20 (60J67 60K35 82B27)},
  MRNUMBER = {4927068},
MRREVIEWER = {Olivier\ Couronn\'{e}},
       DOI = {10.1016/j.spa.2025.104734},
       URL = {https://doi.org/10.1016/j.spa.2025.104734},
}

@article {chelkak-toolbox,
    AUTHOR = {Chelkak, Dmitry},
     TITLE = {Robust discrete complex analysis: a toolbox},
   JOURNAL = {Ann. Probab.},
  FJOURNAL = {The Annals of Probability},
    VOLUME = {44},
      YEAR = {2016},
    NUMBER = {1},
     PAGES = {628--683},
      ISSN = {0091-1798,2168-894X},
   MRCLASS = {60D99 (30C35 39A12 60G50 60K35)},
  MRNUMBER = {3456348},
MRREVIEWER = {Fredrik\ Johansson Viklund},
       DOI = {10.1214/14-AOP985},
       URL = {https://doi.org/10.1214/14-AOP985},
}

@inproceedings {chelkak-icm18,
    AUTHOR = {Chelkak, Dmitry},
     TITLE = {Planar {I}sing model at criticality: state-of-the-art and
              perspectives},
 BOOKTITLE = {Proceedings of the {I}nternational {C}ongress of
              {M}athematicians---{R}io de {J}aneiro 2018. {V}ol. {IV}.
              {I}nvited lectures},
     PAGES = {2801--2828},
 PUBLISHER = {World Sci. Publ., Hackensack, NJ},
      YEAR = {2018},
      ISBN = {978-981-3272-93-4; 978-981-3272-87-3},
   MRCLASS = {82C20 (30G25 35Q82 60J67)},
  MRNUMBER = {3966512},
}

@article {chelkak-semb,
    AUTHOR = {Chelkak, Dmitry},
     TITLE = {Ising model and s-embeddings of planar graphs},
   JOURNAL = {Ann. Sci. \'{E}c. Norm. Sup\'{e}r. (4)},
  FJOURNAL = {Annales Scientifiques de l'\'{E}cole Normale Sup\'{e}rieure.
              Quatri\`eme S\'{e}rie},
    VOLUME = {57},
      YEAR = {2024},
    NUMBER = {5},
     PAGES = {1271--1346},
      ISSN = {0012-9593,1873-2151},
   MRCLASS = {82B20 (05C10 30F99 60J67 81T40 83C45)},
  MRNUMBER = {4885806},
MRREVIEWER = {Olivier\ Couronn\'{e}},
}

@article {chelkak-duminil-hongler,
    AUTHOR = {Chelkak, Dmitry and Duminil-Copin, Hugo and Hongler,
              Cl\'{e}ment},
     TITLE = {Crossing probabilities in topological rectangles for the
              critical planar {FK}-{I}sing model},
   JOURNAL = {Electron. J. Probab.},
  FJOURNAL = {Electronic Journal of Probability},
    VOLUME = {21},
      YEAR = {2016},
     PAGES = {Paper No. 5, 28},
      ISSN = {1083-6489},
   MRCLASS = {60K35 (82B20 82B26)},
  MRNUMBER = {3485347},
MRREVIEWER = {Rapha\"{e}l\ Cerf},
       DOI = {10.1214/16-EJP3452},
       URL = {https://doi.org/10.1214/16-EJP3452},
}

@ARTICLE{CHI-mixed,
       author = {{Chelkak}, Dmitry and {Hongler}, Cl{\'e}ment and {Izyurov}, Konstantin},
        title = "{Correlations of primary fields in the critical Ising model}",
      journal = {arXiv e-prints},
         year = 2021,
        month = mar,
          eid = {arXiv:2103.10263},
        pages = {arXiv:2103.10263},
          doi = {10.48550/arXiv.2103.10263},
archivePrefix = {arXiv},
       eprint = {2103.10263},
 primaryClass = {math-ph},
       adsurl = {https://ui.adsabs.harvard.edu/abs/2021arXiv210310263C}
}

@article {chelkak-izyurov-mahfouf,
    AUTHOR = {Chelkak, Dmitry and Izyurov, Konstantin and Mahfouf, R\'{e}my},
     TITLE = {Universality of spin correlations in the {I}sing model on
              isoradial graphs},
   JOURNAL = {Ann. Probab.},
  FJOURNAL = {The Annals of Probability},
    VOLUME = {51},
      YEAR = {2023},
    NUMBER = {3},
     PAGES = {840--898},
      ISSN = {0091-1798,2168-894X},
   MRCLASS = {82B20 (30G25 60K35 81T40)},
  MRNUMBER = {4583057},
MRREVIEWER = {Hasan\ Ak\i n},
       DOI = {10.1214/22-aop1595},
       URL = {https://doi.org/10.1214/22-aop1595},
}

@article {CLR1,
    AUTHOR = {Chelkak, Dmitry and Laslier, Beno\^{i}t and Russkikh,
              Marianna},
     TITLE = {Dimer model and holomorphic functions on t-embeddings of
              planar graphs},
   JOURNAL = {Proc. Lond. Math. Soc. (3)},
  FJOURNAL = {Proceedings of the London Mathematical Society. Third Series},
    VOLUME = {126},
      YEAR = {2023},
    NUMBER = {5},
     PAGES = {1656--1739},
      ISSN = {0024-6115,1460-244X},
   MRCLASS = {82B20 (30G25)},
  MRNUMBER = {4588721},
       DOI = {10.1112/plms.12516},
       URL = {https://doi.org/10.1112/plms.12516},
}

@article {chelkak-smirnov-ising,
    AUTHOR = {Chelkak, Dmitry and Smirnov, Stanislav},
     TITLE = {Universality in the 2{D} {I}sing model and conformal
              invariance of fermionic observables},
   JOURNAL = {Invent. Math.},
  FJOURNAL = {Inventiones Mathematicae},
    VOLUME = {189},
      YEAR = {2012},
    NUMBER = {3},
     PAGES = {515--580},
      ISSN = {0020-9910,1432-1297},
   MRCLASS = {82B20 (30C20 30C35 60J67 60K35)},
  MRNUMBER = {2957303},
MRREVIEWER = {Ben\ Dyhr},
       DOI = {10.1007/s00222-011-0371-2},
       URL = {https://doi.org/10.1007/s00222-011-0371-2},
}

@article {ising-5authors,
    AUTHOR = {Chelkak, Dmitry and Duminil-Copin, Hugo and Hongler,
              Cl\'{e}ment and Kemppainen, Antti and Smirnov, Stanislav},
     TITLE = {Convergence of {I}sing interfaces to {S}chramm's {SLE} curves},
   JOURNAL = {C. R. Math. Acad. Sci. Paris},
  FJOURNAL = {Comptes Rendus Math\'{e}matique. Acad\'{e}mie des Sciences.
              Paris},
    VOLUME = {352},
      YEAR = {2014},
    NUMBER = {2},
     PAGES = {157--161},
      ISSN = {1631-073X,1778-3569},
   MRCLASS = {82B20},
  MRNUMBER = {3151886},
       DOI = {10.1016/j.crma.2013.12.002},
       URL = {https://doi.org/10.1016/j.crma.2013.12.002},
}

@book {duminil-parafermion,
    AUTHOR = {Duminil-Copin, Hugo},
     TITLE = {Parafermionic observables and their applications to planar
              statistical physics models},
    SERIES = {Ensaios Matem\'{a}ticos [Mathematical Surveys]},
    VOLUME = {25},
 PUBLISHER = {Sociedade Brasileira de Matem\'{a}tica, Rio de Janeiro},
      YEAR = {2013},
     PAGES = {ii+371},
      ISBN = {978-85-8337-014-7},
   MRCLASS = {82B20 (81V35 82C20)},
  MRNUMBER = {3184487},
}

@article {duminil-garban-pete,
    AUTHOR = {Duminil-Copin, Hugo and Garban, Christophe and Pete,
              G\'{a}bor},
     TITLE = {The near-critical planar {FK}-{I}sing model},
   JOURNAL = {Comm. Math. Phys.},
  FJOURNAL = {Communications in Mathematical Physics},
    VOLUME = {326},
      YEAR = {2014},
    NUMBER = {1},
     PAGES = {1--35},
      ISSN = {0010-3616,1432-0916},
   MRCLASS = {81T10 (81P16 81Q10 81Q12 81Q15 81Q30 81Q35 81R12)},
  MRNUMBER = {3162481},
MRREVIEWER = {Rapha\"{e}l\ Cerf},
       DOI = {10.1007/s00220-013-1857-0},
       URL = {https://doi.org/10.1007/s00220-013-1857-0},
}

@article {duminil-hongler-nolin,
    AUTHOR = {Duminil-Copin, Hugo and Hongler, Cl\'{e}ment and Nolin,
              Pierre},
     TITLE = {Connection probabilities and {RSW}-type bounds for the
              two-dimensional {FK} {I}sing model},
   JOURNAL = {Comm. Pure Appl. Math.},
  FJOURNAL = {Communications on Pure and Applied Mathematics},
    VOLUME = {64},
      YEAR = {2011},
    NUMBER = {9},
     PAGES = {1165--1198},
      ISSN = {0010-3640,1097-0312},
   MRCLASS = {60K35 (82B20 82B43)},
  MRNUMBER = {2839298},
MRREVIEWER = {Andrew\ R.\ Wade},
       DOI = {10.1002/cpa.20370},
       URL = {https://doi.org/10.1002/cpa.20370},
}

@article {duminil-manolescu-tassion-fractal,
    AUTHOR = {Duminil-Copin, Hugo and Manolescu, Ioan and Tassion, Vincent},
     TITLE = {Planar random-cluster model: fractal properties of the
              critical phase},
   JOURNAL = {Probab. Theory Related Fields},
  FJOURNAL = {Probability Theory and Related Fields},
    VOLUME = {181},
      YEAR = {2021},
    NUMBER = {1-3},
     PAGES = {401--449},
      ISSN = {0178-8051,1432-2064},
   MRCLASS = {60K35 (82B27)},
  MRNUMBER = {4341078},
       DOI = {10.1007/s00440-021-01060-6},
       URL = {https://doi.org/10.1007/s00440-021-01060-6},
}

@article {garban-wu,
    AUTHOR = {Garban, Christophe and Wu, Hao},
     TITLE = {On the convergence of {FK}-{I}sing percolation to {${\rm
              SLE}(16/3, (16/3)-6)$}},
   JOURNAL = {J. Theoret. Probab.},
  FJOURNAL = {Journal of Theoretical Probability},
    VOLUME = {33},
      YEAR = {2020},
    NUMBER = {2},
     PAGES = {828--865},
      ISSN = {0894-9840,1572-9230},
   MRCLASS = {60J67},
  MRNUMBER = {4091584},
       DOI = {10.1007/s10959-019-00950-9},
       URL = {https://doi.org/10.1007/s10959-019-00950-9},
}

@book {grimmett-random-cluster-book,
    AUTHOR = {Grimmett, Geoffrey},
     TITLE = {The random-cluster model},
    SERIES = {Grundlehren der mathematischen Wissenschaften [Fundamental
              Principles of Mathematical Sciences]},
    VOLUME = {333},
 PUBLISHER = {Springer-Verlag, Berlin},
      YEAR = {2006},
     PAGES = {xiv+377},
      ISBN = {978-3-540-32890-2; 3-540-32890-4},
   MRCLASS = {60K35 (60-02 82-02 82B20 82B43)},
  MRNUMBER = {2243761},
MRREVIEWER = {Olivier\ Garet},
       DOI = {10.1007/978-3-540-32891-9},
       URL = {https://doi.org/10.1007/978-3-540-32891-9},
}

@article {kadanoff-ceva,
    AUTHOR = {Kadanoff, Leo P. and Ceva, Horacio},
     TITLE = {Determination of an operator algebra for the two-dimensional
              {I}sing model},
   JOURNAL = {Phys. Rev. B (3)},
  FJOURNAL = {Physical Review. B. Condensed Matter. Third Series},
    VOLUME = {3},
      YEAR = {1971},
     PAGES = {3918--3939},
      ISSN = {0163-1829},
   MRCLASS = {82.62},
  MRNUMBER = {389111},
MRREVIEWER = {P.\ Caldirola},
}

@ARTICLE{kemppainen-smirnov-ii,
       author = {{Kemppainen}, Antti and {Smirnov}, Stanislav},
        title = "{Conformal invariance in random cluster models. II. Full scaling limit as a branching SLE}",
      journal = {arXiv e-prints},
         year = 2016,
        month = sep,
          eid = {arXiv:1609.08527},
        pages = {arXiv:1609.08527},
          doi = {10.48550/arXiv.1609.08527},
archivePrefix = {arXiv},
       eprint = {1609.08527},
 primaryClass = {math-ph},
       adsurl = {https://ui.adsabs.harvard.edu/abs/2016arXiv160908527K}
}

@phdthesis{mahfouf-phd,
  TITLE = {{Analyse du mod{\`e}le d'Ising planaire au-del{\`a} du cadre critique Z-invariant}},
  AUTHOR = {Mahfouf, R{\'e}my},
  URL = {https://theses.hal.science/tel-03926899},
  NUMBER = {2022UPASM016},
  SCHOOL = {{Universit{\'e} Paris-Saclay}},
  YEAR = {2022},
  MONTH = Aug,
  DOI = {10.70675/c6b256aez6519z461dza770z940f8213cd7e},
  TYPE = {Theses},
  HAL_ID = {tel-03926899},
  HAL_VERSION = {v1},
}

@article {mahfouf-RSW,
    AUTHOR = {Mahfouf, R\'{e}my},
     TITLE = {Crossing estimates for the {I}sing model on general
              s-embeddings},
   JOURNAL = {Proc. Lond. Math. Soc. (3)},
  FJOURNAL = {Proceedings of the London Mathematical Society. Third Series},
    VOLUME = {131},
      YEAR = {2025},
    NUMBER = {4},
     PAGES = {Paper No. e70092, 44},
      ISSN = {0024-6115,1460-244X},
   MRCLASS = {82B20},
  MRNUMBER = {4970632},
       DOI = {10.1112/plms.70092},
       URL = {https://doi.org/10.1112/plms.70092},
}

@ARTICLE{mahfouf-near-critical,
       author = {{Mahfouf}, Remy},
        title = "{The near critical random bond ising model via embedding deformation}",
      journal = {arXiv e-prints},
         year = 2025,
        month = sep,
          eid = {arXiv:2509.08928},
        pages = {arXiv:2509.08928},
          doi = {10.48550/arXiv.2509.08928},
archivePrefix = {arXiv},
       eprint = {2509.08928},
 primaryClass = {math.PR},
       adsurl = {https://ui.adsabs.harvard.edu/abs/2025arXiv250908928M}
}

@ARTICLE{park-maximal,
       author = {{Park}, S.~C.},
        title = "{Conformal Invariance of the FK-Ising Model on Lorentz-Maximal S-Embeddings}",
      journal = {arXiv e-prints},
         year = 2025,
        month = dec,
          eid = {arXiv:2512.10142},
        pages = {arXiv:2512.10142},
          doi = {10.48550/arXiv.2512.10142},
archivePrefix = {arXiv},
       eprint = {2512.10142},
 primaryClass = {math.PR},
       adsurl = {https://ui.adsabs.harvard.edu/abs/2025arXiv251210142P}
}

@incollection {smirnov-icm06,
    AUTHOR = {Smirnov, Stanislav},
     TITLE = {Towards conformal invariance of 2{D} lattice models},
 BOOKTITLE = {International {C}ongress of {M}athematicians. {V}ol. {II}},
     PAGES = {1421--1451},
 PUBLISHER = {Eur. Math. Soc., Z\"{u}rich},
      YEAR = {2006},
      ISBN = {978-3-03719-022-7},
   MRCLASS = {82B20 (30C35 60K35 81T40 82B43)},
  MRNUMBER = {2275653},
MRREVIEWER = {Roland\ M.\ Friedrich},
}

\end{document}